\documentclass[12pt]{article}
\usepackage{graphicx}
\usepackage{amsmath,amsthm,amssymb,amsfonts,enumerate}
\usepackage{euscript,mathrsfs}
\usepackage{xcolor}
\usepackage{empheq}
\usepackage[left=1.5cm,right=1.5cm,top=2cm,bottom=2cm]{geometry}
\usepackage{color}
\usepackage{enumitem}

\usepackage{stmaryrd}
\allowdisplaybreaks
\usepackage[most]{tcolorbox}

\usepackage{epstopdf} 
\usepackage{xcolor} 
\usepackage{graphicx}
\usepackage{bm}

\usepackage[labelformat=simple]{subcaption}

\newcommand{\nc}{\newcommand}

\usepackage{tikz}
\usetikzlibrary{patterns}

\def\softd{{\leavevmode\setbox1=\hbox{d}%
          \hbox to 1.05\wd1{d\kern-0.4ex{\char039}\hss}}}

\catcode`\@=11 \@addtoreset{equation}{section}

\catcode`\@=12

\newcommand{\range}[2]{\{#1, \ldots, #2 \}}
\nc{\pd}{ \partial }
\newcommand\TS{\Delta t}

\newcommand\dt{\mathrm{dt}}
\newcommand\dx{\,\mathrm{d}x}
\newcommand\dS{\,\mathrm{dS}(x)}

\newcommand\Up{\mathrm{Up}}
\newcommand\Fup{F_h^\eps}
\newcommand\bFup{{\bf F}_h^\eps}

\newcommand{\bUp}{\textbf{Up}}
\newcommand{\ith}{ i^{\text{th}}}

\newcommand{\tor}{\mathbb{T}^d}
\newcommand{\bfPhi}{\mathbf{\Phi}}

\nc{\pdedge}{ \eth _{\cal D} }
\nc{\pdedgei}{ \eth _{{\cal D}_i} }
\nc{\pdedgej}{ \eth _{{\cal D}_j} }

\newcommand{\pdduali}{\pdedgei}
\newcommand{\pddualj}{\pdedgej}

\newcommand{\pdmesh}[1]{ \pd _\mesh^{(#1)} }
\newcommand{\pdmeshi}{ \pdmesh{i} }

\newcommand{\pdBii}{ \eth_{\Bii} }
\newcommand{\pdBij}{ \eth_{\Bij} }
\newcommand{\pdBji}{ \eth_{\Bji} }

\newcommand{\bbE}{ \mathbb{E}}
\newcommand{\frakE}{ \mathfrak{E}}

\newcommand{\Grad}{ \nabla _x}
\newcommand{\Gradh}{ \nabla _h}

\newcommand{\GradD}{ \nabla _{\cal D}}
\newcommand{\Gradd}{ \GradD}
\newcommand{\GradB}{ \nabla _{\cal B}}
\newcommand{\Gradq}{ \nabla _Q}

\newcommand{\Gradpiv}{ \nabla^{\Piv}_{\!\! \mesh}}

\newcommand{\Div}{ {\rm div} _x}
\newcommand{\Divw}{ {\rm div} _\mesh^{\bf W}}
\newcommand{\Divq}{ {\rm div} _\mesh^{Q}}
\newcommand\Divh{\mathrm{div}_h}

\newcommand{\Lap}{\Delta_x }

\nc{\shkl}{\sum_{\sigma =K|L\in \faces}}

\nc{\Hc}{ {P} }

\newcommand{\Qh}{Q_h}
\newcommand{\vQh}{{\bf Q}_h}
\newcommand{\vWh}{ {\bf W}_h}
\newcommand{\Whi}{W_{i,h}}

\newcommand{\Piq}{\Pi _Q}
\newcommand{\Piv}{ \Pi _ \E}
\newcommand{\Pid}{ \Pi _{\epsilon}}

\newcommand{\Pivi}{ \Pi _ \E^{(i)}}

\nc{\Dt}{D_t}
\nc{\pdt}{\pd_t}
\nc{\eps}{\varepsilon}

\newcommand{\norm}[1]{\lVert#1\rVert}

\newcommand{\co}[1]{\text{co}\{#1\} }

\newcommand{\avs}[1]{\left\{\!\!\left\{ #1\right\}\!\!\right\}}
\newcommand{\avsi}[1]{\left\{\!\!\left\{ #1\right\}\!\!\right\}^{(i)}}

\newcommand{\ve}{\mathbf{e}}
\renewcommand{\ve}{\bm{e}}
\newcommand{\vei}{\ve_i}

\newcommand\vr{\varrho}
\newcommand\vn{\bm{n}}
\newcommand\vu{\bm{u}}
\newcommand\vm{\bm{m}}
\newcommand{\vM}{\bm{M}}
\newcommand\vv{\bm{v}}

\newcommand\vx{x}

\newcommand{\Td}{\mathbb{T}^d}

\newcommand\vrh{\vr_h}
\newcommand\vuh{\vu_h}
\newcommand\vvh{\vv_h}

\newcommand\vUh{\vU_h}

\newcommand\sumi{\sum_{i=1}^d}
\newcommand\sumj{\sum_{j=1}^d}

\newcommand{\bS}{ \mathbb{S}}
\newcommand{\bI}{\mathbb{I}}
\newcommand{\R}{\mathbb{R}}

\newcommand\mesh{\mathcal{T}}
\newcommand\faces{\mathcal{E}}
\newcommand\edges{\faces}
\newcommand\edgesK{\faces(K)}
\newcommand\edgesi{\faces_i}
\newcommand\edgesiK{\faces_i(K)}

\newcommand\facei{\mathcal{E}_i}

\nc{\E}{\mathcal{E}}
\nc{\calD}{\mathcal{D}}
\nc{\Di}{\mathcal{D}_i}
\nc{\Dj}{\mathcal{D}_j}
\nc{\Bii}{\mathcal{B}_{i,i}}
\nc{\Bij}{\mathcal{B}_{i,j}}
\nc{\Bji}{\mathcal{B}_{j,i}}
\nc{\Eij}{ \widetilde{\E}_{i,j}}
\nc{\Eji}{ \widetilde{\E}_{j,i}}

\newcommand{\neighdual}{ \mathcal{N}^\star }

\newcommand\sumfaceK{\sum_{\sigma \in \edgesK}}
\newcommand\sumfaceiK{\sum_{\sigma \in \edgesiK}}

\newcommand\sumK{\sum_{K \in \mesh}}

\newcommand\sumfaceinti{\sum_{\sigma \in \edgesi}}

\newcommand{\stiei}{\sum_{\sigma \in \edgesi }}

\nc{\intDsi}[1]{ \stiei \int_{D_\sigma}#1\dx}
\nc{\intT}[1]{ \int_0^\tau #1 \dt}
\nc{\intTB}[1]{ \int_0^\tau \left(#1 \right)\dt}

\newcommand\aleq{\lesssim}

\newcommand{\vih}{ v_{i,h} }

\newcommand{\uih}{ u_{i,h} }
\newcommand{\ujh}{ u_{j,h} }

\newcommand{\auih}{ \Ov{\uih} }

\nc{\gradu}{\nabla \vu}

\nc{\bn}{\vn}
\nc{\bfr}{\bm{r}}
\nc{\bfn}{\vn}
\nc{\bfu}{\vu}
\nc{\bfv}{\vv}
\nc{\bfx}{\vx}
\nc{\avu}{\Piq \vu}
\nc{\avuh}{\Piq \vuh}
\nc{\vU}{\bm{U}}

\nc{\avU}{\overline{\bm{U}}}

\newcommand{\sigmap}{\sigma _{K\text{,} i +}}
\newcommand{\sigmam}{\sigma _{K\text{,} i -}}

\nc{\abs}[1]{\left\lvert#1 \right\rvert}
\nc{\jump}[1]{\left\llbracket#1\right\rrbracket}

\newcommand{\Ov}[1]{\overline{#1}}
\newcommand{\bfphi}{\boldsymbol{\phi}}
\nc{\intO}[1]{\int_{\Omega} #1 \dx}
\nc{\intOB}[1]{\int_{\Omega} \left(#1 \right) \dx}
\nc{\intTd}[1]{\int_{\mathbb{T}^d} #1  \dx}
\nc{\intTdB}[1]{\int_{\mathbb{T}^d} \left(#1 \right) \dx}
\nc{\intTO}[1]{\int_0^\tau \int_{\Td} #1 \dx\dt}
\nc{\intTOB}[1]{\int_0^\tau \int_{\Td}\left( #1 \right) \dx\dt}

\nc{\intS}[1]{\int_{\sigma} #1 \dS}
\nc{\intE}[1]{\int_{\faces} #1 \dS}
\nc{\intSh}[1]{\int_{\sigma} #1 \dS}
\nc{\intK}[1]{\int_{K} #1 \dx}

\nc{\ceil}[1]{\lceil#1\rceil}
\nc{\floor}[1]{\lfloor#1\rfloor}

\newtheorem{thm}{Theorem}[section]
\newtheorem{lemma}[thm]{Lemma}

\newtheorem{defi}[thm]{Definition}
\newtheorem{prop}[thm]{Proposition}
\newtheorem{remark}{Remark}
\usepackage[normalem]{ulem}
\usepackage{cancel}
\nc{\mycolor}{\color{magenta}}
\nc{\rdele}{\color{brown}\sout}
\nc{\cblue}[1]{\textcolor{blue}{#1}}
\nc{\cred}{\color{red}}
\nc{\cmag}{\color{magenta}}
\nc{\cbrown}{\color{brown}}
\nc{\mydele}[1]{\textcolor{brown}{\sout {#1}}}

\begin{document}

\title{Improved error estimates for the finite volume  and  the MAC schemes for the compressible Navier--Stokes system\footnotetext{This research was initiated during our ``Research in Pairs" stay at the Mathematisches Forschungsinstitut Oberwolfach in 2021.}}

\author{Eduard Feireisl\thanks{The research of E.F. and B.S. leading to these results has received funding from the
Czech Sciences Foundation (GA\v CR), Grant Agreement 21-02411S. The Institute of Mathematics of the Academy of Sciences of
the Czech Republic is supported by RVO:67985840.\newline
\hspace*{1em} $^\spadesuit$M.L. has been funded by the Deutsche Forschungsgemeinschaft (DFG, German Research Foundation) - Project number 233630050 - TRR 146 as well as by  TRR 165 Waves to Weather. She is grateful to the Gutenberg Research College and Mainz Institute of Multiscale Modelling for supporting her research.}
\and M\' aria Luk\' a\v cov\' a -- Medvi\softd ov\' a$^{\spadesuit}$
 \and Bangwei She$^{*, \clubsuit}$
}

\date{}

\maketitle

\bigskip

\centerline{$^*$ Institute of Mathematics of the Academy of Sciences of the Czech Republic}
\centerline{\v Zitn\' a 25, CZ-115 67 Praha 1, Czech Republic}
\centerline{feireisl@math.cas.cz, she@math.cas.cz}

\bigskip
\centerline{$^\spadesuit$ Institute of Mathematics, Johannes Gutenberg-University Mainz}
\centerline{Staudingerweg 9, 55 128 Mainz, Germany}
\centerline{lukacova@uni-mainz.de}

\bigskip

\centerline{$^{\clubsuit}$Academy for Multidisciplinary studies, Capital Normal University}
\centerline{ West 3rd Ring North Road 105, 100048 Beijing, P. R. China}


\bigskip

\begin{abstract}
We present new error estimates for the finite volume and finite difference methods applied to the compressible Navier--Stokes equations.
The main innovative ingredients of the improved error estimates are a refined consistency analysis combined with a continuous version of the relative energy inequality. Consequently, we obtain better convergence rates than those available in the literature so far. Moreover, the
error estimates hold in the whole physically relevant range of the adiabatic coefficient.
\end{abstract}


\section{Introduction}\label{sec:1}

The Navier--Stokes equations governing the motion of viscous compressible fluids have numerous applications in engineering, physics, meteorology or biomedicine.
In this paper we consider the viscous barotropic fluid endowed, for simplicity, with the isentropic
pressure--density state equation $p= a \vr^{\gamma},$ where
$a > 0$ is a positive constant, and $\gamma >1$ denotes the adiabatic coefficient. The global--in--time existence of weak solutions is known for any $\gamma > \frac{d}{2}$ in the $d$-dimensional
setting, see Lions~\cite{Lions} and \cite{edo}. More recently, Plotnikov and Vaigant \cite{PV} extended
the existence theory for any $\gamma \geq 1$ if $d=2$. Unfortunately,
the multilevel approach used in the existence proof is rather difficult to adapt directly to
a numerical scheme; whence the numerical analysis of the problem remains rather incomplete.

In the last few decades, many efficient and robust numerical methods have been proposed to simulate
the motion of viscous compressible fluid flows. We refer the reader to the  monographs by Dolej\v{s}\'i and Feistauer~\cite{DoFei}, Eymard, Gallou\"et and Herbin \cite{EyGaHe}, Feistauer~\cite{feist1}, Feistauer, Felcman and Stra\v{s}kraba~\cite{feist2}, Toro~\cite{toro}, and the references therein.
Despite a good agreement of the obtained results with experiments, a rigorous
convergence analysis with the associated error estimates have been performed only in a few particular cases.

In his truly pioneering work, Karper \cite{Karper}, see also \cite{feireisl2017numericsbook}, showed convergence (up to a subsequence) of a mixed finite element-finite volume (or discontinuous Galerkin) approximation to a weak solution of the compressible multidimensional Navier--Stokes system under the technical restriction $\gamma > 3$. His proofs basically follows step by step the
existence theory developed in \cite{edo} and as such is difficult to adapt to other numerical methods. Moreover, as the weak solutions are not known to be unique, the result holds up to a subsequence and no convergence rate is available.

Recently, see \cite{FL,FLMS_FVNS,FeLMMiSh}, we have developed a new approach based on the concept of more general dissipative weak (dissipative measure-valued) solution, which, combined with the weak--strong uniqueness and conditional regularity results, yields a rigorous proof of convergence for the mixed finite element-finite volume, finite volume and finite difference Marker-and-Cell (MAC) methods for any $\gamma >1$ as long as the sequence of numerical solution remains uniformly bounded and/or
if the strong solution exists.
The aim of the present paper is to derive error estimates for the finite volume and the MAC methods for full range of the adiabatic coefficient $\gamma >1.$

There are several results concerning error estimates for the compressible Navier--Stokes equations. Under the assumption of the $L^2$-bounds of the discrete derivatives of the numerical solutions, Jovanovi\'{c}~\cite{Jovanovic} studied the convergence rate of a finite volume-finite difference method to the barotropic Navier--Stokes system. In \cite{Liu1, Liu2} Liu analyzed the errors  for $P^k$ conforming finite element method,  $k\geq 2$, assuming the existence of a suitably regular smooth solution.
However, the stability of the method with respect to the discrete energy was not investigated.

Furthermore,  Gallou\"et et al.~\cite{GallouetMAC, Gallouet_mixed} analyzed the {\em unconditional} convergence rates of the mixed finite volume-finite element method  \cite{feireisl2017numericsbook} and  the MAC scheme for $\gamma>3/2$ in the dimension $d=3$.  Similar results have been obtained by Mizerov\'a and She~\cite{MS_MAC}.
All the above mentioned convergence results are based on a discrete version of the relative energy inequality estimating the error between the numerical and the strong solution.  The obtained convergence error is $\mathcal{O}(h^A),$  where $h>0$ is a mesh parameter and
$A = \min\left\{ \frac{2\gamma-3}{\gamma} ,\frac12 \right\}$, cf.~ \cite{GallouetMAC, Gallouet_mixed, MS_MAC}. In particular, the convergence order tends to zero when $\gamma \to \frac{3}{2}$ and remains positive only if $\gamma>3/2.$  Moreover, if $\gamma \geq 2 $, the convergence rate is only $\frac 1 2$ in the energy norm, though the numerical experiments indicate the second order convergence rate.

In view of the existing results, the main novelty of the
present paper is two-fold:
\begin{itemize}
\item  Extending the  error analysis to the full range $\gamma > 1$.
\item  Improving the convergence rate via a detailed consistency and error analysis.
\end{itemize}

Following the strategy proposed in the monograph \cite[Chapter 9]{FeLMMiSh}, we combine the standard consistency errors with the ``continuous'' form of the relative energy inequality. In contrast with
the existing methods based on ad hoc construction of an approximate relative energy inequality, the new approach is rather versatile and free of additional discretization errors.
 In particular, we can handle any consistent energy stable numerical method in the same fashion. We focus on the finite volume method proposed in \cite{FeLMMiSh} and the MAC method from \cite{MS_MAC}. The application to the mixed finite element-finite volume method of Karper~\cite{Karper}
was studied independently and presented in the recent work by
  Novotn\'y and Kwon~\cite{NOKW}.
Compared to the previous results of Gallou\"et et al.~\cite{GallouetMAC, Gallouet_mixed}, we employ the consistency formulation of the numerical solution where the test function is smooth. This new approach avoids the complicated integration by parts formulae on the discrete level and improves the convergence rates of the MAC method presented in \cite{GallouetMAC, MS_MAC}.

The paper is organized in the following way. After presenting the continuous model and the corresponding relative energy, we formulate the numerical schemes: the finite volume and the MAC method,
see Section~\ref{sec:method}. Next, we discuss their energy stability and consistency. The main results on the error estimates are formulated and proved in Section~\ref{error}.

\subsection{Compressible Navier--Stokes system}
We begin with formulating the compressible Navier--Stokes system
\begin{equation}\label{PDE}
\begin{aligned}
\partial_t \vr + \Div(\vr\vu) &= 0,
\\
\partial_t (\vr \vu) + \Div(\vr \vu \otimes \vu) + \Grad  p(\vr) &= \Div \mathbb{S} 
\end{aligned}
\end{equation}
in the time--space cylinder $[0,T] \times \Omega $, $\Omega \subset R^d, d=2,3$, where $\vr$ is the density, $\vu$ is the velocity field, and
$\bS$ is the viscous stress tensor given by
\[\bS = \mu (\Grad \vu + \Grad^T \vu - \frac 2 d  \Div \vu \bI) + \lambda \Div \vu \bI,\; \mu >0, \; \lambda \geq 0. \]
The pressure is assumed to satisfy the {\em isentropic} law
\begin{equation}\label{assumption_p}
p=a \vr^\gamma, \; a>0, \; \gamma>1.
\end{equation}

To avoid technical problems related to a proper numerical
approximation of the physical boundary, we impose the periodic boundary conditions and identify the computational domain with the flat torus $$\Omega=\tor \equiv \left( [0,1] |_{\{ 0,1 \}} \right)^d. $$
The system \eqref{PDE} is supplemented with  finite energy initial data $(\vr_0, \vu_0): \tor \to \R^+ \times \R^d$,
\begin{equation}\label{INI}
\vr(0,\vx) = \vr_0 > 0, (\vr\vu)(0,\vx) =  \vr_0 \vu_0,\ \mbox{ and } E_0 = \intTdB{\frac 1 2  \vr_0 | \vu_0 |^2 + \Hc(\vr_0)} < \infty,
\end{equation}
where
$\Hc $ is the so-called pressure potential, $\Hc(\vr) = \frac{a \vr^{\gamma}}{\gamma-1}$ for the  isentropic gas law \eqref{assumption_p}.

\subsection{Relative energy}

The main tool to evaluate the distance between numerical and strong solutions is the relative energy functional, cf. \cite{FJN}:
\[
\frakE(\vr, \vu| r, \vU) = \intTdB{ \frac12 \vr \abs{\vu- \vU}^2 + \bbE(\vr|r)  },   \mbox{ with } \bbE(\vr|r)= \Hc(\vr) - \Hc'(r) (\vr -r ) -\Hc(r).
\]
As pointed out, relative energy functionals are often used to estimate the distance between a suitable weak solution and the strong solution; whence yielding the weak-strong uniqueness property. Recently, a discrete version of the relative energy has been applied in the error analysis of numerical schemes,  see~\cite{GallouetMAC, Gallouet_mixed,MS_MAC}.

\subsection{Classical solutions}

It will be useful to identify the regularity class of
smooth (classical) solutions to
the Navier--Stokes system~\eqref{PDE} inherited from the initial data~\eqref{INI}. The following result can be the deduced from
\cite[Theorem 3.3]{BrFeHo2016} and \cite[Proposition~2.2]{Hosek}.
\begin{prop}
\label{prop1}
Let the initial data belong to the class
\[
	\vr_0 \in C^3( \mathbb{T}^d), \ \vr_0 > 0 \ \mbox{in}\ \mathbb{T}^d,\
	\vu_0 \in C^3( \mathbb{T}^d; R^d).
\]
Let $(\vr, \vu)$ be a weak solution to problem \eqref{PDE} originating from the initial data \eqref{INI}
such that
\begin{equation}
	\label{r_bound}
	0 \leq \vr \leq \bar{r} \quad \mbox{ and } \quad  | \vu | \leq \bar{u} \mbox{ a.e. in } (0,T) \times \mathbb{T}^d.
\end{equation}

Then $(\vr, \vu)$ is a classical solution of \eqref{PDE}-\eqref{INI} in  $[0,T] \times \mathbb{T}^d$.

If, in addition, $\vr_0$, $\vu_0$ belong to the class
\begin{equation}
\label{INI1}
\vr_0 \in W^{k,2}({\mathbb{T}^d}), \qquad \vu_0 \in  W^{k,2}(\mathbb{T}^d; \R^d), \quad k \geq 6,
\end{equation}
then $\vr \in C([0,T]; W^{k,2}(\mathbb{T}^d))$, $\vu \in C([0,T]; W^{k,2}(\mathbb{T}^d; \R^d))$,
and the following estimate hold
\begin{eqnarray}
\label{bounds_exact}
&& {  \norm{\pd_t^\ell \vr}_{C([0,T] \times \mathbb{T}^d) } +} \| \vr \|_{C^1([0,T] \times \mathbb{T}^d )} +
\| 1/\vr \|_{C([0,T] \times {\mathbb{T}^d})}  + \|  \vr \|_{C([0,T]; W^{k,2}(\mathbb{T}^d) ) }  
\leq D ,\ \ell = 1,2,
\nonumber \\
&&
{  \norm{\pd_t^\ell \vu}_{C([0,T] \times \mathbb{T}^d ;\R^d) } } + \| \vu \|_{C^1([0,T] \times \mathbb{T}^d ;\R^d )} 
  + \| \vu \|_{C([0,T]; W^{k,2}(\mathbb{T}^d; \R^d))}
\leq D, \ \ell = 1,2,
\end{eqnarray}
where $D$ depends solely on $T, \bar{r}, \bar{u}$ and the initial data $(\vr_0, \vu_0)$ via the norm $\| (\vr_0, \vu_0) \|_{W^{k,2}(\mathbb{T}^d; \R^{d + 1})}$ and
$\min_{x \in {\mathbb{T}^d}} \vr_0(x).$
\end{prop}

\begin{proof}
	
The first part was proved in \cite[Proposition~2.2]{Hosek} via the local existence
theory by Valli and Zajaczkowski \cite{VaZa} combined the weak--strong uniqueness principle and the conditional regularity
result by Sun, Wang and Zhang \cite{SuWaZh}. In particular, the bounds \eqref{bounds_exact} were established for
$k = 3$, $\ell = 1$.

Next, as shown in  \cite[Theorem 3.3]{BrFeHo2016}, the solution inherit higher Sobolev regularity
from the data as long as the norm $\| \vu \|_{C([0,T]; W^{2,\infty}(\mathbb{T}^d; \R^d))}$ is controlled. In particular, the estimates \eqref{bounds_exact} can be established. Similarly to Gallagher \cite{Gallagher}, the
proof in \cite{BrFeHo2016} is based on the particular isentropic form of the pressure that enables to
transform the problem to a parabolic perturbation of a symmetric hyperbolic system.

\end{proof}

\section{Numerical methods}\label{sec:method}

First, we introduce suitable notation. By $c$ we denote a positive constant independent of the discretization parameters $\TS$ and $h$.
We shall frequently write $A \aleq B$ if $A \leq cB$ and $A\approx B$ if $A \aleq B$ and $B \aleq A$.
We also write $c\in \co{a,b}$ if $\min(a,b)\leq c\leq\max(a,b)$.
Moreover,  we denote by $\norm{\cdot}_{L^p}$, $\norm{\cdot}_{L^pL^q}$, and $\norm{\cdot}_{L^pW^{q,s}}$ the norms
$\norm{\cdot}_{L^p(\tor)},$ \,  $\norm{\cdot}_{L^p(0,T;L^q(\Td))}$, and $\norm{\cdot}_{L^p(0,T;W^{q,s}(\Td))},$
respectively.

\subsection{Time discretization}\label{ss:time}
We divide the time interval $[0,T]$ into $N_t$ equidistant parts with a fixed time increment $\Delta t$ ($ =T/ N_t$). For a function $f^n$ given at the discrete time instances $t_n = n \Delta t$, $n=0,1,\cdots,N_t$, we define a piecewise constant approximation  $f(t)$ in the following way
\begin{equation*}
f(t,\cdot) = f^0\ \text{ for }\ t < \TS \mbox{ and }  f(t)= f^n \ \text{ for }\ t\in [n\TS,(n+1)\TS), \; n\in\range{1}{N_t}.
\end{equation*}
The time derivative is approximated by the backward Euler method
\begin{equation*}
D_t f = \frac{ f(t,\cdot) - f(t-\TS,\cdot)}{\TS} \quad \mbox{for all } t \in [0,T].
\end{equation*}

\subsection{Space discretization}\label{sec_Space}

To begin, we introduce a uniform structured mesh including primary, dual and bidual grids.
\medskip

\noindent {\bf\emph{Primary grid}}\quad

\noindent We call $\mesh$ the primary grid with the following properties and notations:
\begin{itemize}
\item
The domain $\Td$ is divided into compact  uniform quadrilaterals
$ \Td =\bigcup_{K\in \mesh} K$,
where $\mesh$ is the set of all elements that forms the primary grid.

\item $\edges$ denotes the set of all faces of the primary grid $\mesh$.
Given an element $K\in \mesh$,
 $\mathcal{E}(K)$ is the set of its faces;  $\facei$ is the set of all faces that are  orthogonal to the unit basis vector $\vei$;   $  \edgesiK = \edgesK \cap \facei$ for any $i \in \range{1}{d}$.

\item  $h$  denotes the uniform size of the grid, meaning $|\bfx_K-\bfx_L|= h$ for any neighbouring elements $K$ and $L$, where $x_K$ and $x_L$ are the centers of $K$ and $L$, respectively.

\item $\sigma _{K\text{,} i -}$ and  $\sigma _{K\text{,} i +}$ denote the left and right face of an element $K$ in the $\ith$-direction, respectively.


\item $\mathcal{N}(K)$ denotes the set of all neighbouring elements of $K \in \mesh.$

\item  $\sigma = K|L$ denotes the face $\sigma$ that  separates the elements $K$ and $L$.
Moreover, $\sigma = \overrightarrow{K|L}$ means $\sigma = K|L$ and $\bfx_L - \bfx_K = h\vei$ for some $i \in \range{1}{d}$.

\item $\vn$ denotes the outer normal of a generic face $\sigma$ and $\vn_{\sigma, K}$ denotes the outer normal vector to a face $\sigma \in \edgesK.$
\end{itemize}
%
 {\bf \emph{Dual grid}}\quad

\noindent The dual of the primary grid is determined as follows.
 \begin{itemize}
 \item
 For any face  $\sigma=K|L\in \facei$, a  dual cell is defined as $D_\sigma =D_{\sigma,K} \cup D_{\sigma,L}$, where
 $D_{\sigma,K} = \{\bfx \in K, x_i \in \co{(\bfx_K)_i, (\bfx_\sigma)_i } \}$, see Figure~\ref{fig-grid} for a two dimensional graphic illustration.

\item $\Di =\left\{D_\sigma  \; |  \; \sigma \in \edgesi\right\}$, $i \in \range{1}{d}$, represents the $\ith$ dual grid of $\mesh$.
Note that for each fixed $i \in \range{1}{d}$ it holds
\begin{equation*}
\tor=\bigcup_{\sigma\in {\E_i}} D_{\sigma}, \;\; {\rm int}(D_\sigma)\cap{\rm int}( D_{\sigma'})=\emptyset \mbox{ for } \sigma,\sigma'\in {\E_i},\,\sigma\neq\sigma'.
\end{equation*}

\item
$\widetilde{\E}_i$ is the set of all faces of the $\ith$ dual grid $\Di$ and  $\Eij = \{ \epsilon \in \widetilde{\E}_i | 
\epsilon \mbox{ is orthogonal to } \ve_j\}$.

\item A generic  face of a dual cell $D_\sigma$ is denoted as $\epsilon \in \widetilde{\E}(D_\sigma)$, where  $\widetilde{\E}(D_\sigma)$ denotes the set of all faces of  $D_\sigma.$


\item
$\epsilon = D_\sigma | D_{\sigma'}$ denotes a dual face that separates the dual cells $D_\sigma$ and $D_{\sigma'}$. Moreover,
$\epsilon = \overrightarrow{ D_\sigma | D_{\sigma'} }$ means $\epsilon = D_\sigma | D_{\sigma'}$ and $\bfx_{\sigma'} - \bfx_{\sigma} = h \vei $ 
for  some $i \in \{ 1,\ldots, d\}$.

\item
  $\neighdual(\sigma)$ denotes the set of all faces whose associated dual elements are the neighbours of $D_\sigma,$ i.e.,
\[ \neighdual(\sigma)=\{\sigma'\ |\  D_{\sigma'} \mbox{ is a neighbour of } D_\sigma\}.
\]
\end{itemize}
%
%
 {\bf \emph{Bidual grid}}\quad
 \begin{itemize}
\item
Similarly  to the definition of the dual cell, a bidual cell $D_\epsilon := D_{\epsilon,\sigma} \cap D_{\epsilon,\sigma'}$ associated to $\epsilon = D_\sigma | D_{\sigma'} \in \Eij$ is defined as the union of adjacent halves of
$D_\sigma$ and $D_{\sigma'}$, where $D_{\epsilon,\sigma} = \{x \in D_{\sigma}| x_j \in \co{(x_\sigma)_j, (x_\epsilon)_j } \} $ see Figure~\ref{bidual-grid} for a two dimensional graphic illustration.

\item
$\Bij$ denotes the $j^{\rm th}$ dual grid of $\Di$, that is
set of all bidual cells associated to the bidual faces of $\Eij$. Note that $\Bij = \mesh$ in the case of $i=j$.
 \end{itemize}


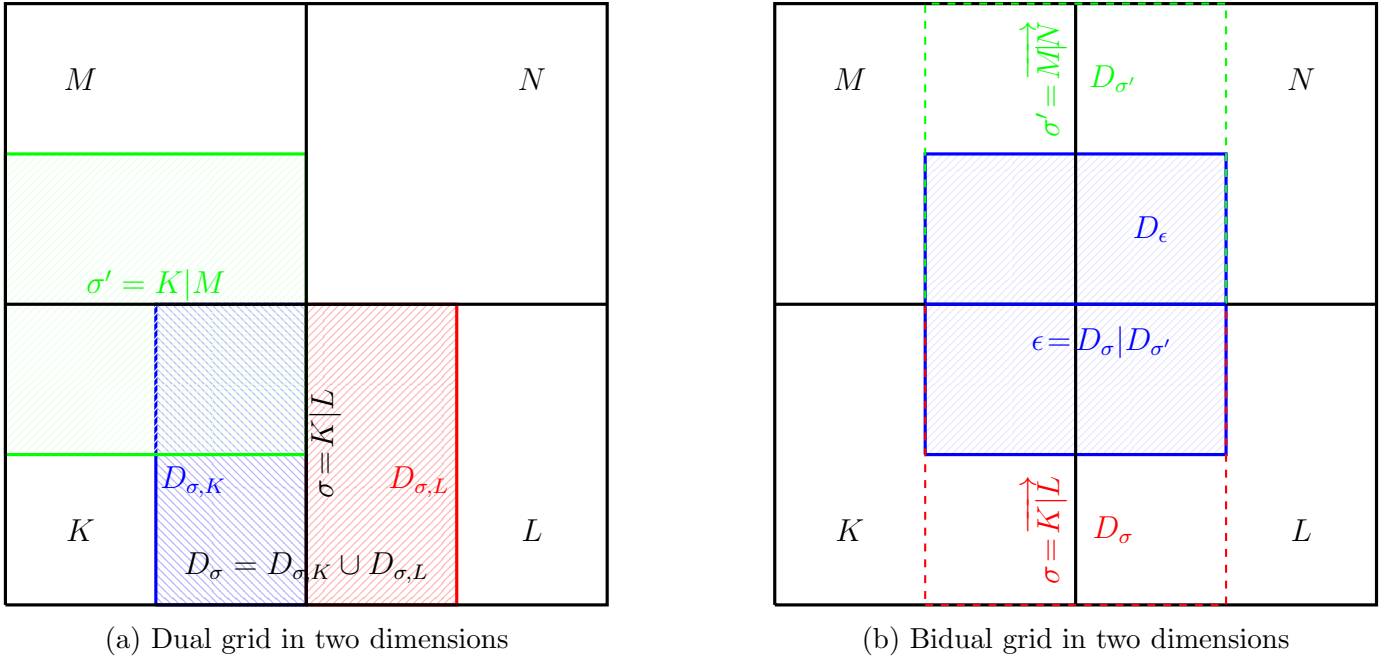
\begin{figure}[hbt]
\centering
\begin{subfigure}{0.45\textwidth}
\centering
\begin{tikzpicture}[scale=1]
\draw[-,very thick,blue=90!, pattern=north west lines, pattern color=blue!30] (2,0)--(4,0)--(4,4)--(2,4)--(2,0);
\draw[-,very thick,red=90!, pattern=north east lines, pattern color=red!30] (4,0)--(6,0)--(6,4)--(4,4)--(4,0);
\draw[-,very thick,green=10!, pattern=north east lines, pattern color=green!10] (0,2)--(4,2)--(4,6)--(0,6)--(0,2);
\path (1,1) node[] { $K$};
\path (7,1) node[] { $L$};
\path (4.25,2.15) node[rotate=90] { $\sigma\!=\!K|L$};
\path (2.5,1.64) node[] { \textcolor{blue}{$D_{\sigma,K}$}} ;
\path (5.5,1.64) node[] { \textcolor{red}{$D_{\sigma,L}$}};
\path (4,0.5) node[] { $D_\sigma=  D_{\sigma,K} \cup D_{\sigma,L}$};
\path (2,4.25) node[] { \textcolor{green}{$\sigma'=K|M$}};
\path (1.,7) node[] { $M$};
\path (7,7) node[] { $N$};%
\draw[-,very thick](0,0)--(8,0)--(8,8)--(0,8)--(0,0);
\draw[-,very thick](4,0)--(4,8);%
\draw[-,very thick](0,4)--(8,4);%
\end{tikzpicture}\caption{Dual grid in two dimensions}\label{fig-grid}
\end{subfigure}
\hfill
\begin{subfigure}{0.45\textwidth}
\centering
\begin{tikzpicture}[scale=1]
\draw[-,very thick,blue=10!, pattern=north east lines, pattern color=blue!10] (2,2)--(6,2)--(6,6)--(2,6)--(2,2);
\path (1,1) node[] { $K$};
\path (7,1) node[] { $L$};
\path (5,5) node[] { \textcolor{blue}{$D_\epsilon$}};
\path (4.5,7) node[] { \textcolor{green}{$D_{\sigma'}$}} ;
\path (4.5,1) node[] { \textcolor{red}{$D_\sigma$}};
\path (3.6,7) node[rotate=90]{ \textcolor{green}{$\sigma'\!=\!\overrightarrow{M\!|\!N}$}};
\path (3.6,1) node[rotate=90] { \textcolor{red}{$\sigma\!=\! \overrightarrow{K|L}$}};
\path (1.,7) node[] { $M$};
\path (7,7) node[] { $N$};%
\draw[-,very thick](0,0)--(8,0)--(8,8)--(0,8)--(0,0);
\draw[-,very thick](4,0)--(4,8);%
\draw[-,very thick](0,4)--(8,4);%
\draw[dashed, thick,green](2,4)--(6,4)--(6,8)--(2,8)--(2,4);%
\draw[dashed, thick,red](2,0)--(6,0)--(6,4)--(2,4)--(2,0);%
\draw[-,very thick,blue](2,4)--(6,4);%
\path (4.35,3.5) node[] { \textcolor{blue}{$\epsilon\!=\! D_\sigma | D_{\sigma'}$} };
\end{tikzpicture}\caption{Bidual grid in two dimensions}\label{bidual-grid}
\end{subfigure}
\caption{MAC grid in two dimensions}\label{MAC-grid}
\end{figure}

\paragraph{Discrete function spaces.}
We  introduce the following spaces of piecewise constant functions:
\begin{equation*}
\begin{aligned}
\Qh &=  \left\{  \phi \mid  \phi_h|_K = \text{ constant } \mbox{for all }\; K\in \mesh \right\}, \qquad \vQh =\Qh^{d}, \\
\vWh & =\left(W_{1,h},\ldots W_{d,h} \right),  \quad \Whi =  \left\{  \phi \mid  \phi_h|_{D_\sigma} = \text{ constant } \mbox{ for all } \; \sigma \in \edgesi  \right\}, \; i\in \range{1}{d}.
\end{aligned}
\end{equation*}
The corresponding projections  read
\begin{align*}
\Piq:& L^1(\tor) \to \Qh, &\
\Piq \phi= \sum_{K\in \mesh} (\Piq \phi)_K  1_{K},
\qquad &\ (\Piq \phi)_K =  \frac{1}{|K|} \int_{K} \phi \dx, \\
\Pivi:& W^{1,1}(\tor) \to \Whi, &\
\Pivi \phi = \sum_{\sigma \in \E}  (\Pivi \phi)_{\sigma} 1_{D_\sigma},
\qquad &\ (\Pivi \phi)_{\sigma} =   \frac{ 1}{|\sigma|} \intSh{ \phi },
\end{align*}
where $1_K$ and $1_{D_\sigma}$  are the characteristic functions.
Further, for any $\bfphi=(\phi_1, \ldots, \phi_d)$ we denote  
$\Piv \bfphi= \left(\Piv^{(1)}\phi_1, \ldots, \Piv^{(d)}\phi_d\right).$
Moreover, for any bidual grid $D_\epsilon $ we define
\begin{equation}\label{pide}
\Pid \phi|_{D_\epsilon} =   \frac{1}{|\epsilon|}  \int_{\epsilon} \phi \dS.
\end{equation}

\subsection{Discrete operators}
\paragraph{Average and jump.}
First, for an piecewise smooth function $f_h$, we define its trace
\[f_h^{\rm out}(x) = \lim_{\delta  \to 0+} f_h(x+\delta \vn) \ \mbox{ and } \ f_h^{\rm in}(x) = \lim_{\delta  \to 0+} f_h(x - \delta \vn).\]
Then for any $r_h \in \Qh$  we define the average operator
\[	\avs{r_h}_\sigma (x) = \frac{r_h^{\rm in}(x) + r_h^{\rm out}(x)}{2}\ \mbox{ for any } x \in \sigma \in \edges.
\]
If in addition, $\sigma  \in \edgesi$ for an $i\in\{1, \ldots, d\}$, we write $\avs{r_h}_\sigma$ as $\avsi{r_h}_\sigma$
and denote
\begin{align*}
\avsi{r_h} = \sum_{\sigma \in \edgesi} 1_{D_\sigma}  \avsi{r_h}_\sigma \; \forall x\in \sigma \in \edges.
\end{align*}
Analogously to the average operator, we define  the jump operator for $r_h \in \Qh$ as
\[
\jump{r_h}_{\sigma}(x) = r_h^{\rm out}(x) -  r_h^{\rm in}(x).
\]
Further, for vector--valued functions $\vvh=(v_{1,h}, \ldots, v_{d,h}) \in \Qh^d$ and $ \vuh=(u_{1,h},\ldots,u_{d,h}) \in \vWh,$ we define
\begin{align*}
\avs{\vvh} &= \left( \avs{v_{1,h}}^{(1)}, \ldots, \avs{v_{d,h}}^{(d)} \right),   \quad
\\
 \auih|_K = \frac{ \uih |_{\sigmap}+ \uih|_{\sigmam} }{2},& \quad
\auih  = \sum_{K \in \mesh} 1_{K}  \auih|_K,   \; \text{ and } \;
\Ov{\vuh}  = \left( \Ov{u_{1,h}},\ldots, \Ov{u_{d,h}} \right).
\end{align*}
Note that for any $\vuh \in \vWh$ we have $\Ov{\vuh}  = \Piq \vuh $.

\paragraph{Gradient operator.}
For any $r_h \in \Qh$ and $\vuh \in \vWh$ we introduce the following gradient operators.
\begin{equation*}
\begin{split}
&  \GradD r_h(\bfx) = \left( \eth_{{\cal D}_1} r_h, \ldots,  \eth_{{\cal D}_d} r_h \right)(\bfx),\\
 & \GradB \vuh (\bfx) = \big( \GradB u_{1,h}(\bfx), \ldots, \GradB u_{d,h}(\bfx)\big) \quad \mbox{with }\
\GradB \uih (\bfx) = \big( \eth_{{\cal B}_{i1}} \uih(\bfx), \ldots, \eth_{{\cal B}_{id}} \uih(\bfx) \big),
\end{split}
\end{equation*}
where
\[
 \pdedgei r_h (\bfx) = \sum_{\sigma \in \edgesi}1_{D_\sigma} (\pdedgei r_h)_{\sigma} ,  \quad  \ (\pdedgei r_h)_{\sigma}  = \frac{r_{L} - r_{K}}{h}, \ \sigma=\overrightarrow{ K|L}\in \edgesi, \;
\]
\[
\pdBij \uih(\bfx) = \sum_{\epsilon \in \Eij } (\pdBij \uih)_{D_\epsilon} 1_{D_\epsilon}, \; (\pdBij \uih)_{D_\epsilon} = \frac{u_{\sigma'} -u_{\sigma}}{h}, \mbox{ for } \epsilon = \overrightarrow{ D_\sigma | D_{\sigma'} } \in \Eij. 
\]
Furthermore,  for any $\vvh \in \vQh$ and $\phi \in W^{1,2}(\tor)$ we set
\begin{eqnarray*}
&&\Gradq \vvh = \sumK 1_K \Gradq \vvh |_K \ \ \mbox{ with }  \ \ \Gradq \vvh |_K = \sumfaceK \frac{|\sigma|}{|K|} \avs{\vvh} \otimes \vn,
\\
&&\Gradpiv \phi = \left(\pdmesh{1} \Piv^{(1)} \phi, \cdots,  \pdmesh{d} \Piv^{(d)}\phi \right).
\end{eqnarray*}
Here, $\pdmeshi$ is defined for  any $\uih \in \Whi$, $i \in \range{1}{d}$ as
\begin{equation*}
\begin{split}
 \pdmeshi \uih(\bfx) = \sumK 1_K (\pdmeshi \uih)_{K} , \quad  \left. \pdmeshi \uih \right \vert_{K} = \frac{ \uih|_{\sigmap} - \uih|_{\sigmam}}{h}, \ K\in\mesh.
\end{split}
\end{equation*}
Note that for any $r_h \in \Qh$ and $\uih \in \Whi$, there hold
\[
\Ov{\pdduali r_h} = \pdmeshi \avsi{r_h}
\quad \mbox{and} \quad \pdBii \uih =\pdmeshi \uih.
\]


\paragraph{Divergence operator.}
For  $\vuh \in \vWh$ and $\vvh \in \vQh$ we define the following discrete divergence operators adjoint to the above discrete gradient operators
\begin{equation*}
\Divw \vuh(\bfx)  = \sumi  \pdmeshi \uih (\bfx) \quad \mbox{and} \quad
\Divq \vvh(\bfx)  =  \sumi  \pdmeshi \avsi{\vih} (\bfx) = \sumi  \Ov{\pdedgei \vih} (\bfx)  .
\end{equation*}
It is easy to observe  for any $\vvh \in \Qh$ that
\begin{equation}\label{diveq}
\Divw \avs{\vvh} = \Divq \vvh.
\end{equation}


\paragraph{Upwind flux.}
Given a velocity field $\vuh \in \vQh \cap \vWh$, the upwind flux function for $r_h \in \Qh$ is given by
\begin{equation*}
\Up [r_h,\vuh]_\sigma 
= r_h^{\rm in} (u_{\sigma})^+ + r_h^{\rm out} (u_{\sigma})^- ,
\end{equation*}
where
\[
r^{\pm} = \frac{1}{2} (r \pm |r|), \quad
u_\sigma= \begin{cases}
\avs{\vuh} \cdot \vn, & \mbox{ if } \vuh \in \vQh ,
\\
\vuh \cdot \vn, & \mbox{ if } \vuh \in \vWh .
\end{cases}
\]
To approximate nonlinear convective terms we apply  the following diffusive upwind flux
\begin{equation*}
\Fup  [r_h,\vuh]_\sigma =  \Up[r_h,\vuh]_\sigma - h^\eps \jump{r_h}_\sigma, \quad  \eps > -1.
\end{equation*}
For $\bfphi_h\in \vQh$ we define a vector-valued upwind flux componentwise
\[
\bUp[\bfphi_h, \vuh] = \left( \Up[\phi_{1,h},\vuh], \cdots, \Up[\phi_{d,h},\vuh] \right), \quad
\bFup[\bfphi_h, \vuh] = \left( \Fup[\phi_{1,h},\vuh], \cdots, \Fup[\phi_{d,h},\vuh] \right) .
\]

\subsection{Preliminary estimates and inequalities}
In this section we present a preliminary material.
%
%
%
First, it is easy to check that the following integration by parts formulae hold, see e.g. \cite[Lemma 2.1]{HS_MAC}.
\begin{lemma}\label{lem_ibp}
Let $r_h, \phi_h \in \Qh,$ and \ $\vuh, \bfphi_h \in \vWh.$ Then
\begin{subequations}\label{IBP}
\begin{equation}\label{IBP2}
 \intTd {r_h \Divw \vuh   } 
= - \intTd{ \vuh \cdot \GradD r_h },
\quad \intTd {r_h  \pdmeshi \uih   }
 = -  \intTd{ \uih  \pdedgei r_h } .
\end{equation}
\end{subequations}
\end{lemma}
\noindent Next, we report the following useful lemmas whose proofs are presented in Appendix~\ref{appa}.
\begin{lemma}\label{L22}
For any $r_h \in Q_h$,  $\vvh \in \vQh$, $\vuh \in \vWh$,  $\psi \in W^{1,2}(\Td)$ and $\vU\in W^{1,2}(\Td;\R^d)$, there hold
\begin{equation}\label{Divcd}
 \intTd{   r_h \Div \vU}
=\intTd{r_h \Divw \Piv \vU},
\end{equation}
\begin{equation}\label{Divcd2}
 \intTd{   \vvh \cdot \Grad \psi }
= \intTd{   \vvh \cdot \Gradpiv \psi }.
\end{equation}
\end{lemma}
\begin{lemma}
\label{L23}
For any $\vuh \in \vWh$, $\vvh \in \vQh$ and $\psi \in W^{1,2}(\Td)$ there hold
\begin{equation}\label{IBP3}
 \intTd{   \vuh \cdot \Grad \psi}
= -  \intTd{ \Pid \psi \; \Divw \vuh } ,
\end{equation}
\begin{equation}\label{IBP4}
 \intTd{   \vvh \cdot \Grad \psi}
= -  \sumi \intTd{ \Pivi \psi   \; \pdedgei \vih  } .
\end{equation}
\end{lemma}
\begin{lemma}\label{L24}
For any $\vuh \in \vWh$, $\vvh \in \vQh$ and $\vU \in W^{2,2}(\Td; \R^d)$, we have
\begin{subequations}
\begin{equation}\label{IBP5}
\intTd{  \avuh \cdot \Lap \vU } =
-   \sumi \sumj \sum_{\epsilon = D_\sigma|D_{\sigma'} \in \Eji } \int_{D_\epsilon} \pdBji \ujh  \left(  \frac{ (\Pivi \pd_i U_j)_{D_\sigma} +(\Pivi \pd_i U_j)_{D_{\sigma'}} }{2}  \right) \dx ,
\end{equation}
\begin{equation}\label{IBP6}
  \intTd{   \vuh \cdot \Grad \Div \vU} =  - \intTd{\Divw \vuh \Pid( \Div \vU) } ,
\end{equation}
\begin{equation}\label{IBP7}
  \intTd{  \vvh \cdot \Lap \vU }  = - \intTd{    \GradD \vvh : \Piv \Grad \vU  },
\end{equation}
\begin{equation}\label{IBP9}
 \intTd{   \avs{\vvh} \cdot \Grad \Div \vU}
=-  \intTd{ \Pid \Div \vU \; \Divq \vvh } .
\end{equation}
\end{subequations}
\end{lemma}

\begin{lemma}\label{NC}
Let $\vvh \in \vQh$, $\vuh \in \vWh$, $\vU \in W^{2,2}(\Td; \R^d)$,
and $\bfPhi \in W^{3,2}(\Td; \R^d)$.
Then for any $i,j\in\range{1}{d}$, we have
\begin{subequations}
\begin{equation}\label{NC1}
\norm{\Piq \vuh -\vuh}_{L^2} \leq \frac{h}2 \norm{\GradB \vuh}_{L^2}, \quad
\norm{ \avs{\vvh} -\vvh}_{L^2} \leq \frac{h}2 \norm{\GradD \vvh}_{L^2},
\end{equation}
\begin{equation}\label{NC2}
 \norm{ \Pivi \pd_i U_j - \pd_i U_j }_{L^2} \leq h \norm{\vU}_{W^{2,2}}
\end{equation}
\begin{equation}\label{NC3}
\norm{  \Div \vU  - \Pid \Div \vU}_{L^2} \leq  h\norm{ \vU}_{W^{2,2}}, \quad
\norm{  \Pid \Div \vU -  \Pivi \Div \vU }_{L^2} \leq h \norm{ \vU}_{W^{2,2}}.
\end{equation}
\begin{equation}\label{NC4}
\norm{  \Grad \Div \bfPhi - \Gradq  \Divh \Piq \bfPhi}_{L^2} \leq  h\norm{ \bfPhi}_{W^{3,2}}, \quad
\norm{  \Lap \bfPhi -   \Divw \Gradd \Piq \bfPhi }_{L^2} \leq h \norm{ \bfPhi}_{W^{3,2}}.
\end{equation}
\end{subequations}
\end{lemma}

\subsection{Finite volume and finite difference methods}

We proceed by presenting a finite volume and a finite difference numerical method that will be used to approximate the Navier--Stokes system \eqref{PDE}--\eqref{INI}. Both methods have been already successfully applied in numerical simulations, see, e.g., \cite{FeLMMiSh}. In our recent work \cite{FLMS_FVNS, FeLMMiSh, MS_MAC}, the convergence was shown for $\gamma > 1$ via the concept of  dissipative measure-valued solutions. However, the error analysis was missing for the finite volume method and suboptimal for the finite difference method.

\subsubsection{Finite volume method}
We introduce the finite volume (FV) method approximating the Navier--Stokes system  \eqref{PDE}--\eqref{INI}.
\begin{defi}[FV scheme]\label{defVFV}
Given the initial data \eqref{INI}, we set $(\vrh^0,\vrh^0 \vuh^0) =(\Piq\vr_0, \Piq[\vr_0 \vu_0])$.
The \emph{FV approximation}
$(\vrh^n , \vuh^n) \in \Qh \times \vQh,$ $n=1,\dots, N,$ of the Navier--Stokes system \eqref{PDE}--\eqref{INI} is a solution of the following system of algebraic equations:
\begin{subequations}\label{VFV_S}
\begin{equation}
\intTd{ D_t \vrh^n  \phi_h } -   \intE{  \Fup [\vrh^n ,\vuh^n ]
\jump{\phi_h}   } = 0 \quad \mbox{for all } \ \phi_h \in \Qh,  \label{VFV_SD}
\end{equation}
\begin{equation} \label{VFV_SM}
\begin{aligned}
&\intTd{ D_t  (\vrh^n  \vuh^n ) \cdot \bfphi_h } -   \intE{ \bFup [\vrh^n  \vuh^n ,\vuh^n ]
\cdot \jump{\bfphi_h}   } - \intTd{  p_h^n    \Divh \bfphi_h   }
\\&= - \mu \intTd{ \Gradd \vuh^n    :  \Gradd \bfphi_h  }
- \nu  \intTd{\Divq   \vuh^n   \; \Divq \bfphi_h }
\quad \mbox{for all } \ \bfphi_h \in \vQh,
\end{aligned}
\end{equation}
where $\nu = \frac{d-2}{d} \mu + \lambda.$
\end{subequations}
\end{defi}
\medskip

\subsubsection{Finite difference MAC method}
We proceed by presenting the finite difference MAC scheme that is based on a staggered grid approach.
On the one hand,  the discrete density $\vrh$ and pressure $p_h=p(\vrh)$ are approximated on the primary grid $\mesh$.   On the other hand,  the $\ith$ component of the velocity field $\uih$ is approximated on the $\ith$ dual grid $\Di$.
The MAC scheme reads as follows.
\begin{defi}[MAC scheme]\label{defMAC}
Given the initial data \eqref{INI}, we consider
 $(\vrh^0,\vrh^0 \Piq\vuh^0) =(\Piq\vr_0, \Piq[\vr_0 \vu_0]).$
 The \emph{MAC approximation} of the Navier--Stokes system  \eqref{PDE}--\eqref{INI}  is a sequence
$ (\vrh^n ,\vuh^n  ) \in \Qh \times \vWh,$ $n=1,2,\dots,N,$
which solves  the following system of algebraic equations:
\begin{subequations}\label{MAC_S}
\begin{equation}\label{MAC_SD}
\intTd{ D_t \vrh^n  \phi_h } -   \intE{  \Fup [\vrh^n ,\vuh^n ]\jump{\phi_h}   }  = 0  \ \ \mbox{ for all }\  \phi_h \in \Qh,\\
\end{equation}
\begin{equation} \label{MAC_SM}
\begin{aligned}
&\intTd{ D_t  (\vrh^n   \avuh^n ) \cdot \Ov{\bfphi_h} } -   \intE{ \bUp [\vrh^n   \avuh^n ,\vuh^n ] \cdot \jump{\Ov{\bfphi_h} }   }
 \\& + \mu \intTd{ \GradB \vuh^n    :  \GradB \bfphi_h  } + \nu  \intTd{\Divw   \vuh^n   \; \Divw \bfphi_h }
  - \intTd{  p_h^n    \Divw \bfphi_h   }
 \\ & =
- h^{\eps+1}  \sumi  \sumj  \intTd{      \avs{\auih^n }^{(j)}  (\pddualj \vrh)     \pddualj  \Ov{  \phi_{i,h} }  },
\mbox{ for all } \bfphi_h = (\phi_{1,h}, \ldots, \phi_{d,h})\in \vWh,
\end{aligned}
\end{equation}
where $\nu = \frac{d-2}{d} \mu + \lambda.$
\end{subequations}
\end{defi}
In what follows, we will denote by $\vrh(t), \vuh(t)$  the piecewise constant approximations of $\vrh^n, \vuh^n$, $n=0,1, \dots, N$ on the time interval $[0,T]$, see Section~\ref{ss:time}.
We note that both methods,  the FV method \eqref{VFV_S} as well as  the MAC method \eqref{MAC_S}, preserve the positivity of density and   conserve the mass
\begin{equation}\label{MC}
\vrh(t)>0\  \mbox{ and } \intTd{\vrh(t)} =  M  \quad \mbox{ for all } \; t\in(0,T),
\end{equation}
where
$ M  := \intTd{\vr_0} $ denotes the fluid mass, see e.g. \cite[Lemma 11.2]{FeLMMiSh}.

\subsection{Energy stability}
The essential feature of any numerical scheme is its stability. We now recall the energy stability of both numerical methods introduced above, see \cite[Theorem 11.1 and 14.1]{FeLMMiSh}
\begin{lemma}[Energy estimates]\label{thm_stability}
Let $(\vrh, \vuh)$ be a numerical solution obtained either by the FV scheme \eqref{VFV_S} or by the MAC scheme \eqref{MAC_S} with $\gamma>1$. Then for all $\tau\in(0,T)$, it holds
\begin{equation}\label{ST}
\intTd{ \left( \frac12 \vrh  \abs{\Piq \vuh}^2 + \Hc(\vrh) \right)(\tau) }   + \mu \int_0^\tau \intTd{ {\abs{\Gradh \vuh}^2 } } \dt  + \nu \int_0^\tau \intTd{ \abs{\Divh \vuh }^2 } \dt
\leq   E_0 ,
\end{equation}
where  $ E_0=\intTd{\Big(\frac12 \vr_0 |\vu_0|^2 + \Hc(\vr_0)  \Big)}$ is the initial energy
and
\begin{equation*}
(\Gradh \vuh, \Divh \vuh) = \begin{cases}
(\Gradd \vuh, \Divq \vuh) & \mbox{ for } \vuh \in \vQh \mbox{ in the FV scheme};
\\
(\GradB \vuh, \Divw \vuh) & \mbox{ for } \vuh \in \vWh \mbox{  in the MAC scheme}.
\end{cases}
\end{equation*}
Moreover, there exists $c>0$ which may depend on the fluid mass $ M $ and the initial energy $E_0$ but is independent of the parameters $h$ and $\TS$ such that
\begin{subequations}\label{ests1}
\begin{align}
&
\norm{\vrh \abs{\avuh}^2}_{L^{\infty}L^1}    \leq  c  , \quad
\norm{\vrh}_{L^{\infty}L^\gamma}   \leq  c  , \quad
 \norm{\vrh\avuh}_{L^\infty L^{\frac{2\gamma}{\gamma+1}}}    \leq  c  ,
\label{est_ener}
\\&
\norm{\Divh \vuh}_{L^2L^2}   \leq  c  , \quad
\norm{\Gradh \vuh}_{L^2 L^2}    \leq  c  ,\quad
\norm{\vuh}_{L^{2}L^6}    \leq  c  .
\label{est_u}
\end{align}
\end{subequations}
\end{lemma}
\medskip

\subsection{Consistency formulation}\label{subsec:con}
The next important ingredient of our approach is the consistency formulation of the numerical scheme.
\begin{lemma}[Consistency formulation]\label{thm_CS}
Let $(\vrh, \vuh)$ be either a solution of the FV scheme \eqref{VFV_S} or the MAC scheme \eqref{MAC_S}  with $\Delta t \approx h \in(0,1)$, $\gamma>1$ and $\eps >-1$.

Then for all $\tau \in (0,T)$, $\phi \in 
L^\infty(0,T; W^{2,\infty}(\Td)),$
 $\pd_t ^2 \phi \in L^\infty ((0,T) \times \Td)$
 and $\bfphi \in 
 L^\infty(0,T; W^{2,\infty}(\Td;  \R^d))$,
$\pd_t ^2 \bfphi \in L^\infty ((0,T) \times \Td; \R^d)$
   there holds
\begin{subequations}\label{CS}
\begin{equation} \label{CS1}
\left[ \intTd{ \vrh \phi  } \right]_{t=0}^\tau =
\int_0^\tau \intTdB{  \vrh \partial_t \phi + \vrh \Piq \vuh \cdot \Grad \phi } \dt  \;  + e_\vr(\tau, \TS, h, \phi),
\end{equation}
\begin{multline} \label{CS2}
\left[  \intTd{ \vrh \Piq \vuh \cdot \bfphi } \right]_{t=0}^\tau =
\int_0^\tau \intTdB{  \vrh \avuh \cdot \partial_t \bfphi + \vrh \avuh \otimes \avuh  : \Grad \bfphi  + p_h \Div \bfphi } \dt
\\
 -  \mu \int_0^\tau \intTd{  \Gradh \vuh : \Grad \bfphi} \dt
 -  \nu \int_0^\tau \intTd{  \Divh \vuh \; \Div \bfphi} \dt
\; +e_{\vm}(\tau, \TS, h, \bfphi),
\end{multline}
where the consistency errors are bounded as follows:
\begin{equation}\label{CS3}
\begin{aligned}
&\abs{e_\vr(\tau, \TS, h, \phi) } \leq \begin{cases}
  C_\vr \big(\TS  + h + h^{1+\eps} + h^{1+\beta_D}   \big) & \mbox{ for the FV method }
\\    C_\vr \big(\TS + h^{1+\eps}  + h^{1+\beta_D}  \big) & \mbox{ for the MAC method }
\end{cases}
\\
&\abs{e_{\vm}(\tau, \TS, h, \bfphi) } \leq
\begin{cases}
C_{\bm{m}}  \big( \sqrt{\TS} + h + h^{1+\eps}  + h^{1+\beta_M}     \big)& \mbox{ for the FV method }
\\    C_{\bm{m}}   \big( \sqrt{\TS} +h + h^{1+\eps}  + h^{1+\beta_M} + h^{1+\eps + \beta_D }   \big) & \mbox{ for the MAC method.}
\end{cases}
\\
\end{aligned}
\end{equation}
Here,  the constant $C_\vr$ depends on
$$  \mbox{ the initial energy } E_0, T, \mbox { and } \norm{\phi}_{ L^\infty(0,T; W^{2,\infty}(\Td)) },\,  \norm{\pd_t^2 \phi}_{L^\infty((0,T) \times \Td) },
$$
and $C_{\bm{m}}$ depends on
$$  E_0, T,   \norm{\bfphi}_{ L^\infty(0,T; W^{2,\infty}(\Td; \R^d) },\,
\norm{\pd_t^2 \bfphi}_{L^\infty((0,T) \times \Td; \R^d) }
.$$
Further, the exponents $\beta_D$ and $\beta_M$ are given by
\begin{equation}\label{betas}
\begin{aligned}
\left. \begin{array}{l}
\beta_D =
\begin{cases}
 \max\left\{ - \frac{3\eps+3+d}{6\gamma}, \frac{\gamma-2}{2\gamma}d \right\}, & \mbox{if } \gamma \in(1,2), \\
0, &\mbox{if } \gamma \geq 2,
\end{cases}
\end{array}
\right.
\quad
\beta_M =
\begin{cases}
- \frac{3\eps+3+d}{6\gamma},
& \mbox{ if } \gamma \in(1,2), \\
\frac{\gamma-3}{3\gamma}d, &\mbox{ if } \gamma \in [2,3), \\
0, &\mbox{ if } \gamma \geq 3 \mbox{ for } d=3,\\
0, &\mbox{ if } \gamma > 2 \mbox{ for } d=2.
\end{cases}
\end{aligned}
\end{equation}

\end{subequations}
\end{lemma}
\begin{remark}
Consistency formulation for the FV and MAC method was introduced in \cite[Theorem~11.2]{FeLMMiSh} and \cite[Theorem~14.2]{FeLMMiSh}, respectively.
Instead of an abstract consistency error identified in \cite{FeLMMiSh}, Lemma~\ref{thm_CS} provides  an explicit bound in terms of the numerical step and regularity of the associated test function.
Moreover, we improve the result of \cite{FeLMMiSh} by requiring less regularity of the test functions.
\end{remark}
\begin{proof}[Proof of Lemma~\ref{thm_CS}]
The consistency errors arising from the time derivative term can be evaluated in the following way. First, by a direct calculation, we obtain

\begin{equation}\label{TAU2}
\begin{aligned}
&\int_0^{t^{n+1}} \intTd{ D_t r_h(t) \Piq \varphi(t) } \dt =
\int_0^{t^{n+1}} \intTd{ \frac{r_h(t)- r_h(t-\Delta t) } {\Delta t} \varphi(t) } \dt
\\ &
= \frac{1}{\TS}\int_0^{t^{n+1}} \intTd{ r_h(t) \varphi(t) } \dt -
\frac{1}{\Delta t}\int_{-\Delta t}^{t^n} \intTd{ r_h(t) \varphi(t+ \Delta t) } \dt
\\ &=  - \int_0^{t^{n+1}} \intTd{ r_h(t) D_t \varphi(t+\TS) } \dt
+ \frac{1}{\Delta t}\int_{ t^n }^{t^{n+1}} \intTd{ r_h(t) \varphi(t +\Delta t) } \dt
\\& \quad  - \frac{1}{\Delta t}\int_{-\Delta t}^{0} \intTd{ r_h(t) \varphi(t+ \Delta t) } \dt
\\ & = - \int_0^{t^{n+1}} \intTd{ r_h(t) D_t \varphi (t+\TS) } \dt
+ \frac{1}{\Delta t}\int_{t^n}^{t^{n+1}} \intTd{ r_h(t) \varphi(t) } \dt
- \frac{1}{\TS} \int_0^{\TS} \intTd{ r_h^0 \varphi (t) } \dt
\\ &= - \int_0^{t^{n+1}} \intTd{ r_h(t) \pd_t \varphi (t) } \dt + \intTd{  \underbrace{r_h(\tau)}_{= r_h^n \, \forall \tau \in [t^n,t^{n+1}) } \varphi (\tau)} - \intTd{ r_h^0 \varphi (0)}  + I_1 + I_2 +I_3,
\end{aligned}
\end{equation}
for any $\tau \in [t_n, t_{n+1})$, $n=1,\dots, N_T$,
where
\begin{align*}
I_1 &= \intTd{  r_h^0 \frac{1}{\TS} \int_0^{\TS} \left( \varphi (0) - \varphi(t)\right) \dt }
	\aleq \TS \norm{\pdt \varphi}_{L^\infty L^\infty}  \norm{r_h^0}_{L^1},
\\
I_2 &= \intTd{  \frac{1}{\TS} \int_{t^n}^{t^{n+1}}   \big( r_h(t) \varphi (t+\Delta t) - r_h(\tau)\varphi(\tau)\big) \dt }
	\\& = \intTd{  \frac{1}{\TS} \int_{t^n}^{t^{n+1}}   r_h^n  \big( \varphi (t+\Delta t) - \varphi(\tau) \big) \dt }
	\aleq \TS  \norm{r_h^n}_{L^1} \norm{\pdt \varphi}_{L^\infty L^\infty},
\\
I_3 &= \int_0^{t^{n+1}} \intTd{ r_h(t) \left(\pd_t \varphi (t) - D_t\varphi(t+\TS)   \right)} \dt
	 =   \intTd{  \sum_{k=0}^{n} \int_{t^k}^{t^{k+1}}   r_h(t) \big(\pd_t \varphi (t) - D_t\varphi(t+\TS)   \big)  \dt}
	\\& \leq   \TS \norm{ \pd_{t}^2\varphi   }_{L^\infty L^\infty}  \norm{r_h}_{L^\infty L^1}.
\end{align*}
Collecting the above estimates we obtain from \eqref{TAU2} that
\begin{equation}\label{CST}
 \left[\intTd{ r_h \varphi}\right]_0^{\tau} 
-  \int_0^{t^{n+1}} \intTdB{ D_t r_h(t) \Piq \varphi(t) + r_h(t) \pd_t \varphi (t) } \dt
\leq   \TS \norm{   \pd_{t}^2\varphi  }_{L^\infty L^\infty}  \norm{r_h}_{L^\infty L^1},
\end{equation}
whenever $\tau \in [t_n, t_{n+1})$,
where $r_h$ stands for $\vrh$ or $\vrh \vuh.$ 

Analogously as in the proofs of \cite[Theorem~11.2]{FeLMMiSh} and \cite[Theorem~14.2]{FeLMMiSh}, we obtain
\begin{subequations}
\begin{equation} \label{CSTA}
\left[ \intTd{ \vrh \phi  } \right]_{t=0}^\tau =
 \int_0^{t^{n+1}} \intTdB{  \vrh \partial_t \phi + \vrh \Piq \vuh \cdot \Grad \phi } \dt \;  + e_\vr(\tau, \TS, h, \phi),
\end{equation}
\begin{equation} \label{CSTB}
\begin{aligned}
 \left[ \intTd{ \vrh \Piq\vuh \cdot \bfphi  } \right]_{t=0}^\tau
&= \int_0^{t^{n+1}} \intTdB{  \vrh \Piq\vuh \cdot \partial_t \bfphi + \vrh \Piq \vuh \otimes \Piq \vuh : \Grad \bfphi } \dt
\\& + \int_0^{t^{n+1}} \intTd{ \big( p_h \mathbb{I}
 -  \mu  \Gradh \vuh  - \nu   \Divh \vuh \big): \Grad \bfphi} \dt
 \; +e_{\vm}^* (\tau, \TS, h, \phi),
\end{aligned}
\end{equation}
\end{subequations}
where $e_{\vm}^*$ is controlled by
\[
\abs{e_{\vm}^*(\tau, \TS, h, \bfphi) } \leq
\begin{cases}
C_{\bm{m}}  \big( \TS + h + h^{1+\eps}  + h^{1+\beta_M}     \big)& \mbox{ for the FV method ,}
\\    C_{\bm{m}}   \big( \TS +h + h^{1+\eps}  + h^{1+\beta_M} + h^{1+\eps + \beta_D }   \big) & \mbox{ for the MAC method.}
\end{cases}
\]
In order to derive \eqref{CS1} it suffices to realize that the time integral from $\tau$ to $t^{n+1}$ at the right hand side of \eqref{CSTA} is of order $\mathcal{O}(\Delta t)$ . Indeed
\begin{equation}\label{CSTR}
\begin{aligned}
 &\int_\tau^{t^{n+1}} \intTd{ \Big( \vrh \partial_t \phi + \vrh \Piq \vuh \cdot \Grad \phi \Big) } \dt
 \\&
 \leq \big( \norm{\pdt \phi}_{L^\infty L^\infty} \norm{\vrh^n}_{L^\infty L^1} + \norm{\Grad \phi}_{L^\infty L^\infty}\norm{  \vrh^n \Piq \vuh^n}_{L^\infty L^1} \big)  \int_\tau^{t^{n+1}} 1 \dt
 \aleq \TS.
\end{aligned}
\end{equation}
Combining \eqref{CSTA} and \eqref{CSTR} yields \eqref{CS1}.

Similarly, to get \eqref{CS2} we need the following estimate
\begin{align*}
&
\int_\tau^{t^{n+1}} \intTd{ \Big( \vrh \avuh \cdot \partial_t \bfphi + \big(\vrh \avuh \otimes \avuh   + p_h \mathbb{I}
 -  \mu  \Gradh \vuh  - \nu   \Divh \vuh \big): \Grad \bfphi \Big) } \dt
 \\& \leq
\int_\tau^{t^{n+1}} \norm{ \vrh^n\avuh^n}_{L^1(\tor)} \norm{\pdt \bfphi}_{L^\infty(\tor) } \dt
+\int_\tau^{t^{n+1}} \norm{  \vrh^n \avuh^n \otimes \avuh^n   + p_h^n \mathbb{I}    }_{L^1(\tor)} \norm{\Grad \bfphi}_{L^\infty(\tor) }  \dt
 \\& +\int_\tau^{t^{n+1}} \norm{   \mu  \Gradh \vuh^n  + \nu   \Divh \vuh^n  }_{L^1(\tor)} \norm{\Grad \bfphi}_{L^\infty(\tor) }  \dt
  \\& \leq
\TS \norm{ \vrh^n\avuh^n}_{L^\infty L^1} \norm{\pdt \bfphi}_{L^\infty L^\infty }
+\TS  \norm{  \vrh^n \avuh^n \otimes \avuh^n   + p_h^n \   }_{L^\infty L^1} \norm{\Grad \bfphi}_{L^\infty L^\infty }
 \\& + \norm{\Grad \bfphi}_{ L^\infty  L^\infty }  \norm{   \mu  \Gradh \vuh^n  + \nu   \Divh \vuh^n  }_{L^2 L^1}  \left(\int_\tau^{t^{n+1}}  1^2\dt \right)^{1/2}  \aleq  \sqrt{\TS}.
\end{align*}
Substituting the above estimate into \eqref{CSTB} we obtain \eqref{CS2}, which completes the proof.
\end{proof}



\begin{lemma}[Consistency formulation for a bounded numerical solution]\label{thm_CSB}

Let the assumptions of \linebreak Lemma~\ref{thm_CS} hold. Moreover, let $\vrh$ and $\vuh$ be uniformly bounded, i.e., there exist positive constants $\Ov{\vr}$ and $\Ov{u}$ such that
\begin{equation}\label{bdd}
\vrh \leq \Ov{\vr} \mbox{ and } \abs{\vuh} \leq \Ov{u}.
\end{equation}

 Then for all $\tau \in (0,T)$,  $\phi \in  L^\infty(0,T; W^{2,\infty}(\Td))$,
 $\pd_t ^2 \phi \in L^\infty ((0,T) \times \Td)$
 and $\bfphi \in L^\infty(0,T; W^{2,\infty}(\Td;  \R^d)) \cap L^2(0,T; W^{3,2}(\tor;  \R^d))$, $\pd_t ^2 \bfphi \in L^\infty ((0,T) \times \Td; \R^d)$,
  there holds
\begin{subequations}\label{CSB}
\begin{equation} \label{CSB1}
\left[ \intTd{ \vrh \phi  } \right]_{t=0}^\tau =  \int_0^\tau
 \intTdB{  \vrh \partial_t \phi + \vrh \Piq \vuh \cdot \Grad \phi  } \dt   \;  + e_\vr(\tau, \TS, h, \phi),
\end{equation}
\begin{multline} \label{CSB2}
\left[  \intTd{ \vrh \Piq \vuh \cdot \bfphi } \right]_{t=0}^\tau =
\int_0^\tau  \intTdB{  \vrh \avuh \cdot \partial_t \bfphi + \vrh \avuh \otimes \avuh  : \Grad \bfphi  + p_h \Div \bfphi } \dt
\\
 +   \int_0^\tau \intTd{   \vuh \cdot ( \mu \Lap \bfphi  +   \nu  \Grad \Div \bfphi )} \dt
\; +e_{\vm}(\tau, \TS, h, \bfphi),
\end{multline}
where the consistency errors can be bounded as follows
\begin{equation}\label{CSB3}
\begin{aligned}
\abs{e_\vr(\tau, \TS, h, \phi) } \leq   C_\vr (\TS + h  ),
\quad \abs{e_{\vm}(\tau, \TS, h, \bfphi) } \leq C_{\bm{m}}  ( \TS + h  )
\end{aligned}
\end{equation}
Here,  the constant $C_\vr$ depends on
$$ \Ov{\vr},  \Ov{u},  E_0, T,\norm{\phi}_{ L^\infty(0,T; W^{2,\infty}(\Td)) },\,
 \norm{\pd_t^2 \phi}_{L^\infty((0,T) \times \Td) },
$$
and $C_{\bm{m}}$ depends on
$$ \Ov{\vr},\Ov{u}, E_0, T,  \norm{\bfphi}_{ L^\infty(0,T; W^{2,\infty}(\tor; \R^d)) },\,   \norm{\bfphi}_{ L^2(0,T;W^{3,2}(\tor; \R^d)) }, \norm{\pd_t^2 \bfphi}_{L^\infty((0,T) \times \tor; \R^d) } .$$
\end{subequations}
\end{lemma}
\begin{proof}
We will present the proof for the FV method, the proof for the MAC method is analogous.
First, we denote the errors of the inviscid fluxes as
\begin{equation}\label{e1}
e_1 =   \int_0^{t^{n+1}} \intTd{  \vrh \Piq \vuh \cdot \Grad \phi } \dt
 -  \int_0^{t^{n+1}}  \intE{  \Fup [\vrh ,\vuh ] \jump{ \Piq \phi }   }   \dt ,
\end{equation}
\begin{equation}\label{e2}
\begin{aligned}
 e_2=&  \int_0^{t^{n+1}}    \intTd{   \vrh \avuh \otimes \avuh  : \Grad \bfphi   } \dt
- \int_0^{t^{n+1}}    \intE{ \bFup [\vrh  \vuh ,\vuh ]  \cdot \jump{\Piq \bfphi}   } \dt
\\& + \int_0^{t^{n+1}}    \intTd{  p_h \Div \bfphi -  p_h    \Divh \Piq \bfphi  } \dt.
\end{aligned}
\end{equation}
Analogously as in the proof of \cite[Theorem 11.3]{FeLMMiSh} we get
\begin{equation*}
 |e_1| \leq c(\norm{\phi}_{L^\infty W^{2, \infty}}) h \norm{\vrh}_{L^2L^2}\ \
\mbox{ and } \ \  |e_2| \leq c(\norm{\bfphi}_{L^\infty W^{2, \infty}}) h \norm{\vrh \vuh}_{L^2L^2} .
\end{equation*}
In view of assumption \eqref{bdd} the errors $e_1$ and $e_2$ are controlled by
\begin{equation}\label{e12}
\begin{aligned}
& |e_1| \leq c(\norm{\phi}_{L^\infty W^{2, \infty}}) h \norm{\vrh}_{L^2L^2} \leq c(\norm{\phi}_{L^\infty W^{2, \infty}}, \Ov{\vr}) h ,
\\&  |e_2| \leq c(\norm{\bfphi}_{L^\infty W^{2, \infty}}) h \norm{\vrh \vuh}_{L^2L^2} \leq c(\norm{\bfphi}_{L^\infty W^{2, \infty}}, \Ov{\vr}, \Ov{u}) h.
\end{aligned}
\end{equation}
Now, summing up \eqref{e1} and  \eqref{CST} with $r_h =\vrh$, and recalling the estimates \eqref{CSTR} and \eqref{e12} implies \eqref{CSB1}.
Moreover, summing up \eqref{e2} and \eqref{CST} with $r_h =\vrh \vuh$ we get

\begin{equation}\label{em}
\begin{aligned}
 \left[ \intTd{\vrh  \vuh  \cdot \bfphi} \right]_0^{\tau} = &
\int_0^\tau \intTd{  \vrh \vuh \cdot \partial_t \bfphi + \big(\vrh \vuh \otimes \vuh   + p_h \mathbb{I}  \big): \Grad \bfphi
 +    \vuh \cdot (\mu \Lap \bfphi + \nu   \Grad \Div \bfphi)  } \dt
 \\& + e_2 +e_3 +e_4,
\end{aligned}
\end{equation}
where $e_2$ is given in \eqref{e2}. The error terms $e_3$ and $e_4$ can be estimated in the following way
\begin{align*}
|e_3| & =\left| -  \int_0^{t^{n+1}} \intTd{  \vuh \cdot( \mu \Lap \bfphi + \nu \Grad \Div \bfphi)
+ \big( \mu \Gradd \vuh :  \Gradd \Piq \bfphi + \nu \Divh \vuh \Divh \Piq \bfphi \big) }\right|
\\
&= \left| \int_0^{t^{n+1}} \intTd{ \mu \vuh \cdot  (\Divw \Gradd \Piq \bfphi - \Lap \bfphi ) + \nu \vuh \cdot ( \Gradq \Divh \Piq \bfphi  - \Grad \Div \bfphi ) }\right|
\\& \leq c(\norm{\bfphi}_{L^2W^{3,2}}, \Ov{u} ) h,
\\
|e_4 | &=
\left|\int_\tau^{t^{n+1}} \intTd{  \vrh \avuh \cdot \partial_t \bfphi + \big(\vrh \avuh \otimes \avuh   + p_h \mathbb{I}\big): \Grad \bfphi
+ \vuh \cdot (\mu \Lap \bfphi + \nu   \Grad \Div \bfphi)
} \dt\right|
\\&
\leq \TS \norm{\bfphi}_{C^1}\left(\norm{\vrh\Piq \vuh}_{L^\infty L^1} +\norm{\vrh\abs{\Piq \vuh}^2 }_{L^\infty L^1} + \norm{p_h}_{L^\infty L^1}\right) +
\Ov{u} \norm{\bfphi}_{L^\infty W^{2,\infty}} \int_\tau^{t^{n+1}} \dt
\\& \leq c(\norm{\bfphi}_{L^\infty W^{2,\infty}},  \norm{\bfphi}_{C^1}, \Ov{\vr},\Ov{u}) \TS.
\end{align*}
Consequently, collecting the estimates of $e_2$, $e_3$ and $e_4$ we observe that \eqref{CSB2} follows from \eqref{em}, which completes the proof.
\end{proof}

\section{Error estimates}
\label{error}

 This section is the heart of the paper. We prove the main result -- the convergence rates for the FV \eqref{VFV_S}
 and MAC \eqref{MAC_S} schemes. If, in addition, the numerical solutions are uniformly bounded, the convergence rates can be improved to the first order.

\begin{thm}[Convergence rates]\label{thm_EE}
Let $\gamma > 1$ and the initial data $(\vr_0, \vu_0)$ satisfy
$$
\vr_0 \in W^{k,2}({\tor}),\ \vr_0 > 0 \mbox{ in } \tor, \qquad \vu_0 \in  W^{k,2}(\tor; \R^d), \quad
k \geq 6.
$$
Suppose that the Navier--Stokes system \eqref{PDE} admits a classical solution $(\vr,\vu)$ defined on $[0,T] \times \tor$, with the initial data $(\vr_0, \vu_0).$
Further, let $(\vrh, \vuh)$ be a numerical solution obtained either by the FV scheme \eqref{VFV_S} or by the MAC scheme \eqref{MAC_S} emanating from the projected initial data $(\vrh^0, \vuh^0)$.

Then  there exists a positive number
$$ c=c( T, \|(\vr_0, \vu_0)\|_{W^{k,2}(\Td; \R^{d+1})},
\inf \vr_0, \norm{(\vr, \vu)}_{C([0,T]\times \Td; \R^{d+1})})$$
such that
\begin{equation}\label{RATE}
\begin{aligned}
&\sup_{0\leq t \leq T}\frakE(\vr_h,\vu_h | \vr, \vu)
+\mu  \int_0^T \intTd{  |\Gradh \vuh-  \Grad \vu|^2 } \dt
+ \nu \int_0^T \intTd {|\Divh \vuh - \Div \vu|^2 } \dt
 \\
&\leq c ( h^A + \sqrt{\TS}),
\end{aligned}
\end{equation}
\begin{equation}\label{RATRM}
\begin{aligned}
&\norm{\vrh -\vr}_{L^\infty L^\gamma} + \norm{\vr_h \vu_h - \vr \vu}_{L^\infty L^\frac{2 \gamma}{\gamma+1}}
\aleq c(\sqrt{\TS} +h)^{1/2} + c(\sqrt{\TS} +h^A)^{1/\gamma} & \quad \mbox{for } \gamma \leq 2,
\\&
\norm{\vrh - \vr}_{L^\infty L^2} + \norm{\vr_h \vu_h - \vr \vu}_{L^\infty L^\frac{2 \gamma}{\gamma+1}}
\aleq c(\sqrt{\TS} +h^A)^{1/2}  &\quad \mbox{for } \gamma > 2,
\end{aligned}
\end{equation}
and
\begin{equation}\label{RATU}
 \norm{\vuh - \vu}_{L^2 L^2} \aleq c(\sqrt{\TS} +h^A)^{1/2} .
\end{equation}
The convergence rate $A$ reads
\begin{equation}\label{constA}
A=  \begin{cases}
A_{FV} :=
\min \left\{ 1, 1+\eps, 1+\beta_D, 1+\beta_M \right\} & \mbox{ for the FV method},
\\
A_{MAC} :=
\min \left\{ 1, 1+\eps, 1+\beta_D,   1+\beta_M, 1+\eps+\beta_D \right\} & \mbox{ for the MAC method}.
\end{cases}
\end{equation}
Here the constants $\beta_D$ 
and $\beta_M$ are given in \eqref{betas}.
\end{thm}

\begin{remark}
Let us discuss the obtained convergence rate $\mathcal{O}(h^A)$
for the choice	$\TS =h$ and different values of $\gamma >1$, $d=2,3.$
\begin{itemize}
\item For the case  $d=2$, we obtain the following convergence rate $A$:
\begin{itemize}
\item Let $\gamma \geq 2$.  Then for any $\eps \geq 0$ both numerical methods have the first order convergence rate, i.e.  $A=1.$
\item Let $\gamma \in (1,2)$. The convergence rates are different for  the FV and MAC schemes.
\begin{itemize}
\item
$A_{FV}= \min\left\{   1- \frac{5+3\eps}{6\gamma}, 1, 1+\eps \right\}$. Choosing the optimal value of $\eps$,  $\eps= - \frac{5}{3+6\gamma}\in(-\frac59, -\frac13)$,  the convergence rate $A_{FV}=1 + \eps$ varies between
$\frac49$ for $\gamma \searrow 1$ and $\frac23$ for $\gamma \nearrow 2.$
\item
$A_{MAC}= \min\left\{   1- \frac{5+3\eps}{6\gamma}, 1, 1+\eps, 1+\eps-\frac{5+3\eps}{6\gamma} \right\}$ reaches its maximum value
 $\frac{6\gamma-5}{6\gamma}>0$ at $\eps =0$. Thus, the convergence rate varies between $\frac 1 6$ for $\gamma \searrow 1$ and $\frac{7}{12}$ for $\gamma \nearrow 2.$
\end{itemize}
\end{itemize}
\item For the case  $d=3$, we obtain the following convergence rate $A$:
\begin{itemize}
\item Let $\gamma \geq 3$. Then for any $\eps \geq 0$ both methods have first order convergence rates, i.e.~$A=1$.
\item Let $\gamma \in [2,3)$. Then for any $\eps \geq \frac{2\gamma-3}{\gamma}$ we have $A= \frac{2\gamma-3}{\gamma}$ and the convergence rate varies between $\frac 1 2$  for $\gamma = 2$ and 1 for $\gamma \nearrow 3$.
\item Let $\gamma \in (1,2)$. \begin{itemize}
\item
$A_{FV}= \min\left\{   1- \frac{2+\eps}{2\gamma}, 1, 1+\eps \right\}$. Choosing an optimal value of $\eps,$  $\eps= - \frac{ 2}{1+2\gamma}\in(-\frac23, -\frac25)$, $A_{FV} =  1+\eps$ and varies between $\frac13$ for $\gamma \searrow 1$ and  $\frac35$ for $\gamma \nearrow 2.$
\item
$A_{MAC}= \min\left\{   1- \frac{2+\eps}{2\gamma}, 1, 1+\eps, 1+\eps-\frac{2+\eps}{2\gamma} \right\} $ reaches its maximum value  $\frac{\gamma-1}{\gamma} > 0$ at $\eps =0$.  Note that $A_{MAC}$ varies between $0$ when $\gamma \searrow 1$ and $\frac 1 2$ when $\gamma \nearrow 2.$
\end{itemize}
\end{itemize}
\end{itemize}
\end{remark}
\begin{remark}
In view of the above results, the convergence rates available in the literature, see e.g. \cite{GallouetMAC,Gallouet_mixed,MS_MAC}, are not optimal. Indeed,  for $d=3$ and $\gamma=\frac32$, they
degenerate to $0.$ Moreover,  no error analysis  is available for $\gamma < \frac 3 2.$
Our approach yields error estimates also for $\gamma \in(1, \frac 3 2].$  In addition, we have better convergence rates, e.g.,  for $d=3$ and $\gamma=\frac32$, where the convergence errors are $\mathcal{O}(h^{\frac 3 4 })$ and $\mathcal{O}(h^{\frac 1 3 })$ for the FV and MAC schemes, respectively.
\end{remark}
\bigskip

\begin{proof}[Proof of Theorem~\ref{thm_EE}]
First, by a straightforward but lengthy calculation, see Appendix~\ref{appd},
we observe the following relative energy inequality
\begin{equation}\label{RE1}
\begin{aligned}
& \left[  \frakE(\vrh, \vu_h| \vr , \vu )   \right]_0^\tau +  \int_0^\tau \intTdB{\mu \abs{ \Gradh \vuh}^2 +  \nu \abs{\Divh \vuh}^2 } \dt
\\& \leq \int_0^\tau \intTdB{\vrh  \pdt  \frac{\abs{\vu}^2}2 + \vrh \avuh \cdot \Grad \frac{ \abs{\vu}^2}2 }\dt  \; +   e_\vr \left(\tau, \TS, h, \abs{\vu}^2/2 \right)
\\& - \int_0^\tau \intTdB{\vrh   \pdt \Hc'(\vr) + \vrh \avuh \cdot \Grad \Hc'(\vr)}   \;  -   e_\vr( \tau, \TS, h,\Hc'(\vr) )
\\& - \int_0^\tau \intTdB{\vrh \avuh \cdot \pdt \vu + \vrh \avuh \otimes \avuh : \Grad \vu + p_h \Div \vu} \dt
\\& + \int_0^\tau \intTdB{ \mu  \Gradh \vuh: \Grad \vu + \nu \Divh \vuh \Div \vu} \dt  \;   +  {e_{\vm} (\tau, \TS, h,-\vu)}
\\& + \int_0^\tau \intTd{ \pdt \big( \vr \Hc'(\vr) - \Hc(\vr) \big) } \dt .
\end{aligned}
\end{equation}
Next, we observe the following identities
\[\begin{aligned}
&\vrh \avuh \cdot \Grad \frac{\abs{\vu}^2}{2} - \vrh \avuh \otimes \avuh : \Grad \vu
\\& = - \vrh (\avuh-\vu) \otimes (\avuh-\vu) : \Grad \vu
-\vrh (\avuh -\vu) \cdot (\vu \cdot\Grad \vu),
\end{aligned}
\]
\[
\Hc''(\vr) =\frac{1}{\vr} p'(\vr),  \quad  \vr \Hc'(\vr) - \Hc(\vr)  =  p(\vr), \quad
\pdt( \vr  \Hc'(\vr) - \Hc(\vr))  = \pdt p(\vr).
\]
Then by substituting the above equalities into \eqref{RE1} and denoting
\[  e_S=    e_\vr \left(\tau, \TS, h,\abs{\vu}^2/2\right) -
e_\vr(\tau, \TS, h, \Hc'(\vr) ) +  e_{\vm} (\tau, \TS, h,-\vu),
\]
we obtain
\begin{equation}\label{RE2}
\begin{aligned}
& \left[ \frakE(\vrh, \vu_h| \vr , \vu)  \right]_0^{\tau} +  \int_0^\tau \intTdB{\mu \abs{\Gradh \vuh- \Grad \vu}^2 +  \nu \abs{\Divh \vuh- \Div \vu}^2 } \dt
\\& \leq  e_S +
 \int_0^\tau \intTd{\vrh  (\vu - \avuh) \cdot (\pdt \vu + \vu \cdot \Grad \vu) } \dt
 \\ &
  - \intTO{ \vrh (\avuh-\vu) \otimes (\avuh-\vu) : \Grad \vu}
\\&
+ \mu \intTOB{  \abs{\Grad  \vu}^2 -  \Gradh \vuh: \Grad \vu }
\\&
+ \nu \intTOB{    \abs{\Div \vu}  ^2  - \Divh \vuh \Div \vu    }
\\ &
 + \intTOB{   \pdt p(\vr)  - \vrh  \frac{ \pdt p(\vr)}{\vr} - \vrh \avuh \cdot \frac{\Grad p(\vr)}{\vr} -  p_h \Div \vU} .
\end{aligned}
\end{equation}
As $(\vr,\vu)$ satisfies the Navier--Stokes system \eqref{PDE}, we know that
\[ \vr (\pdt \vu + \vu \cdot \Grad \vu) = \mu \Lap \vu + \nu \Grad \Div \vu - \Grad p(\vr).\]
Substituting this equality into \eqref{RE2} we get
\begin{align*}
& \left[ \frakE(\vrh, \vu_h| \vr, \vu)  \right]_0^{\tau}  +  \intTOB{\mu \abs{\Gradh \vuh- \Grad \vu}^2 +  \nu \abs{\Divh \vuh- \Div \vu}^2 }
\\ & \leq e_S
+ \intTO{(\vrh-\vr)  (\vu - \avuh) \cdot (\pdt \vu + \vu \cdot \Grad \vu) },
 \\  &
   - \intTO{ \vrh (\avuh-\vu) \otimes (\avuh-\vu) : \Grad \vu} ,
\\  &
+  \mu \intTOB{  \abs{\Grad  \vu}^2 -  \Gradh \vuh: \Grad \vu  +(\vu - \avuh) \cdot \Lap \vu }
\\  &
+  \nu \intTOB{     \abs{\Div \vu}  ^2  - \Divh \vuh \Div \vu+(\vu - \avuh) \cdot \Grad \Div \vu}
\\   &
+ \intTOB{     \frac{\vr -\vrh}{\vr } \pdt p(\vr) - \frac{\vrh }{\vr}\avuh \cdot \Grad p(\vr) -  p_h \Div \vu}  - \intTO{ (\vu - \avuh) \cdot \Grad p(\vr)  }.
\end{align*}

Rearranging the terms on the right hand side, we arrive at
\begin{equation*}
 \left[ \frakE(\vrh, \vu_h| \vr , \vu)  \right]_0^\tau +  \intTOB{\mu \abs{\Gradh \vuh- \Grad \vu}^2 +  \nu \abs{\Divh \vuh- \Div \vu}^2 }
\leq e_S + \sum_{i=1}^5 R^E_i ,
\end{equation*}
where the integrals $R^E_i, i=1,\cdots,5$, read
\begin{align*}
 R^E_1&= \intTO{(\vrh- \vr)  (\vu - \avuh) \cdot (\pdt \vu + \vu \cdot \Grad \vu + \frac{\Grad p(\vr)}{\vr}) }
 \\ &=  \intTO{(\vrh-\vr)  (\vu - \avuh) \cdot  \frac{\Div \bS(\Grad \vu)}{\vr} }
 \\ R^E_2 &=  - \intTO{ \vrh (\avuh-\vu) \otimes (\avuh-\vu) : \Grad \vu} ,
\\ R^E_3 & =  - \mu \intTOB{  \Gradh \vuh: \Grad \vu  + \avuh \cdot \Lap \vu },
\\ R^E_4 & = -   \nu \intTOB{   \Divh \vuh \Div \vu + \avuh \cdot \Grad \Div \vu},
\\  R^E_5 &=  - \intTO{\big(p_h - p'(\vr) (\vrh - \vr) -p(\vr) \big) \Div \vu} .
\end{align*}
Next, for $i=1,\cdots, 5$ we analyze $R^E_i$ such that it can be controlled either by the relative energy or the mesh parameter $h$.

\paragraph{Term $R^E_1$.}
Applying H\"older's inequality and Lemma \ref{lmRU} we obtain
\begin{align*}
\abs{R^E_1} &\leq \frac{1}{\underline{r}} \norm{ \Div \bS(\Grad \vu)}_{L^\infty( (0,T)\times \tor )} \left(C_0 \intT{  \frakE(\vrh, \vuh| \vr, \vu)  }  + C_1 \delta \norm{\Gradh \vuh - \Grad \vu}_{L^2}^2 + C_2 \delta h^2\right)
\\& = C_0^* \intT{  \frakE(\vrh, \vuh| \vr, \vu)  }  + C_1^* \delta \norm{\Gradh \vuh - \Grad \vu}_{L^2}^2 + C_2^* \delta h^2,
\end{align*}
where
$C_0^*>0$ depends on $\norm{\vu}_{L^\infty W^{2,\infty}}$, $\norm{\vr}_{C([0,T] \times \tor)},  M , E_0, \gamma$,  $\delta$,  and $\underline{r} = \min_{[0,T] \times \tor} \vr$; \\
$ C_1^*>0$ depends on $\norm{ \vu}_{L^\infty W^{2,\infty}}$, $ M , E_0$, and $\gamma$;\\
$C_2^*>0$ depends on $\norm{ \vu}_{L^\infty W^{2,\infty}}$, $ M $, $E_0, \gamma$, and $\norm{\Grad \vu}_{L^\infty( (0,T)\times \tor )} $.

\paragraph{Term $R^E_2$.}
Thanks to H\"older's inequality we observe the following estimate.
\begin{align*}
\abs{R^E_2} \leq  C \intT{ \frakE(\vrh, \vuh| \vr, \vu)  },
\end{align*}
where $C$ depends on $\norm{\Grad \vu}_{L^\infty( (0,T)\times \tor)} $.

\paragraph{Term $R^E_3$.} We analyze the third term  $R^E_3$ in two cases.

First, we consider the case of the FV scheme. In this case $\vuh \in \vQh$,
 $\Gradh \vuh =\GradD \vuh$ and $\Piq \vuh = \vuh$. Thus,
\begin{align*}
\abs{ R^E_3} & =   \mu  \abs{ \intTOB{   \Gradh \vuh: \Grad \vu  + \avuh \cdot \Lap \vu } }
 = \mu \abs{ \intTOB{   \GradD \vuh: \Grad \vu  + \vuh \cdot \Divw \Piv \Grad \vu }}
 \\&=
  \mu \abs{ \intTO{   \GradD \vuh: ( \Grad \vu -  \Piv \Grad \vu) }}
  \leq  \mu h \norm{\Gradh \vuh}_{L^2L^2} \norm{\vu}_{L^2 W^{2,2}} ,
\end{align*}
where we have used the equality \eqref{Divcd}, the integration by parts formula \eqref{IBP2}, and the estimate \eqref{NC2}.

Second, we consider the case of $\vuh \in \vWh$ obtained by the MAC scheme.
In this case, $\Gradh \vuh =\GradB \vuh$ and the term $R^E_3$ can be estimated in the following way
\begin{align*}
\abs{ R^E_3} & =   \mu  \abs{ \intTOB{   \Gradh \vuh: \Grad \vu  + \avuh \cdot \Lap \vu } }
\\& = \mu \abs{ \intT{   \sumi  \sumj  \sum_{\epsilon = D_\sigma|D_{\sigma'} \in \Eji } \int_{D_\epsilon} \pdBji \ujh  \left(\pd_i U_j   - \frac{ (\Pivi \pd_i u_j)_{D_\sigma} +(\Pivi \pd_i u_j)_{D_{\sigma'}} }{2}  \right) \dx }
}
   \\&
  \leq \mu h \norm{\Gradh \vuh}_{L^2L^2} \norm{\vu}_{L^2 W^{2,2}} ,
\end{align*}
where we have applied \eqref{IBP5},  H\"older's inequality and the estimate \eqref{NC2}.

Consequently, we have for both cases
\[
\abs{ R^E_3} \leq Ch,
\]
where the constant $C$ depends on $\mu,$ the initial energy $E_0$ and $\norm{\vU}_{L^2 W^{2,2}}$.

\paragraph{Term $R^E_4$.} We analyze the term $R^E_4$ also in two cases.

First, for $\vuh \in \vQh$ obtained by the FV method, we have
$ \Divh \vuh  = \Divq \vuh$, $\Piq \vuh =\vuh$ and thus
\begin{align*}
\abs{R^E_4}  &=  \nu \abs{\intTOB{    \Divh \vuh \Div \vu + \avuh \cdot \Grad \Div \vu} }
\\& =  \nu \abs{\intTOB{    \Divq \vuh \Div \vu + \vuh \cdot \Grad \Div \vu  -\big( \Pid \Div \vu \; \Divq \vvh + \avs{\vvh} \cdot \Grad \Div \vu  \big) } }
\\&=  \nu \bigg|  \intTO{  \Big(  \Divq \vuh   ( \Div \vu  - \Pid \Div \vu)  + (\vuh - \avs{\vuh}) \cdot \Grad \Div \vu  \Big)}   \bigg|
\\&
\leq  h \left(  \norm{\Divh \vuh}_{L^2L^2}+ \norm{\Gradh \vuh}_{L^2L^2} \right) \norm{\vu}_{L^2 W^{2,2}},
\end{align*}
where we have used the identity \eqref{IBP9}, H\"older's inequality, the estimates \eqref{NC1} and \eqref{NC3}.

Second, for the  case of $\vuh \in \vWh$ we have
\begin{align*}
\abs{R^E_4}  &=  \nu \abs{\intTOB{    \Divh \vuh \Div \vu + \avuh \cdot \Grad \Div \vu} }
\\&=  \nu \abs{ \intTOB{    \Divw \vuh \Div \vu - \big(  \Divw \vuh \Pid \Div \vu + \vuh \cdot \Grad \Div \vu \big) + \avuh \cdot \Grad \Div \vu}   }
\\&=  \nu \abs{ \intTOB{    \Divw \vuh   ( \Div \vu  - \Pid \Div \vu) + (\avuh - \vuh) \cdot \Grad \Div \vu}   }
\\&
\leq   \nu h \left( \norm{\Divh \vuh}_{L^2L^2}+ \norm{\Gradh \vuh}_{L^2L^2} \right)  \norm{\vu}_{L^2 W^{2,2}},
\end{align*}
where  \eqref{IBP6}, H\"older's inequality, the estimates \eqref{NC1} and \eqref{NC3} were applied.

Consequently, we have for both cases
\[
\abs{ R^E_4} \leq Ch,
\]
where the constant $C$ depends on $\nu$, initial energy $E_0$, and $\norm{\vu}_{L^2 W^{2,2}}$.

\paragraph{Term $R^E_5$.}
The estimate of $R^E_5$ is straightforward by applying H\"older's inequality, i.e.,
\begin{align*}
 \abs{R^E_5}  \leq  C \intT{ \frakE(\vrh, \vuh| \vr, \vu)  } ,
\end{align*}
where $C$ depends on $\norm{\Div \vu}_{L^\infty((0,T)\times \tor)}$.

Consequently, collecting the above estimates of $R^E_i$ for $i=1,\cdots, 5$, we find
\begin{equation}\label{RE3}
\begin{aligned}
&  \frakE(\vrh, \vu_h| \vr, \vu)(\tau) +  \intTOB{(\mu -   C_1^* \delta) \abs{\Gradh \vuh- \Grad \vu}^2 +  \nu \abs{\Divh \vuh- \Div \vu}^2 }
\\& \leq  e_S   + \frakE(\vrh, \vuh| r, \vu)(0)   +  C_0^* \intT{ \frakE(\vrh, \vuh| r, \vu)  } +  C_2^* \delta h^2.
\end{aligned}
\end{equation}

Applying the standard projection error estimates we get
\begin{equation}\label{eini}
\frakE(\vrh, \vu_h| \vr, \vu)(0) \leq C h^2,
\end{equation}
where $C$ depends on $\norm{\vr_0}_{C}$ and $\norm{\vu_0}_{L^2 W^{1,2}}$.

Consequently, by choosing $\delta < \frac {\mu}{C_1^*}$, substituting \eqref{eini} into \eqref{RE3}, using Gronwall's lemma and recalling the consistency error \eqref{CS3}, we may infer that
\begin{equation*}\label{RE4}
  \frakE(\vrh, \vu_h| \vr, \vu)(\tau) +  \intTOB{ \abs{\Gradh \vuh- \Grad \vu}^2 +   \abs{\Divh \vuh- \Div \vu}^2 }
 \leq C e^{\frac{\tau C_0^*}{1-\TS C_0^*}}
 (\sqrt{\TS} + h^A)
\end{equation*}
for $\TS < \frac{1}{C_0^*}$. Here, the constant $C$ depends on $\norm{\vr}_{L^{\infty} W^{2,\infty} }, \norm{ \vu}_{L^{\infty} W^{2,\infty} }   $ and the exponent $A$ is given by \eqref{constA}.

Finally, we combine the above estimate  with Lemma~\ref{lem_EN} and Lemma~\ref{lmSP2} in order to obtain \eqref{RATRM} and \eqref{RATU}, respectively. Note that $E_0$ and $M$ are bounded by  the  norm $\|(\vr_0, \vu_0)\|_{W^{k,2}(\Td; R^{d + 1})}  .$
Due to Proposition~\ref{prop1} all terms depending on the norms of the  exact solution $(\vr,\vu)$ as well as $\underline{r}$  are bounded by a constant
$c=c( T, \|(\vr_0, \vu_0)\|_{W^{k,2}(\Td; R^{d + 1})},
\norm{(\vr, \vu)}_{C([0,T]\times \Td; R^{d + 1})})$ which
finishes the proof.
\end{proof}

Finally, we observe that  under the assumption that the numerical solutions $(\vrh,\vuh)$ are uniformly bounded, the above error estimates can be improved.
Indeed, applying Lemma~\ref{thm_CSB}, Lemma~\ref{lem_EN} and Lemma~\ref{lmSP2} we derive the first order error rate.
\begin{thm}[Error rates for bounded numerical solutions]\label{thm_EEB}
In addition to the hypotheses of Theorem~\ref{thm_EE},  let the numerical solution $(\vrh, \vu_h)$ be uniformly bounded,
\begin{equation}\label{bdda}
\| \vrh \|_{L^\infty((0,T) \times \Td)} \leq \Ov{\vr} \quad \mbox{ and } \quad  \| \vuh \|_{L^\infty((0,T) \times \Td; \R^d)} \leq \Ov{u}.
\end{equation}

Then   there exists a positive number
$$c=c \left(   T, \norm{( \vr_0,  \vu_0)}_{W^{k,2}(\Td; \R^{d +1})},  \inf{\vr_0}, \Ov{\vr}, \Ov{u},  \right)
%
$$
such that
\begin{equation*}
\begin{aligned}
&\sup_{0\leq t \leq \tau}\frakE(\vr_h,\vu_h | \vr, \vu)
+\mu     \intTO{|\Gradh \vuh-  \Grad \vu|^2 }
+ \nu    \intTO{|\Divh \vuh - \Div \vu |^2 }
 \\
&\leq c ( h + \TS)
\end{aligned}
\end{equation*}
for all $\tau \in[0,T]$,
and
\begin{equation*}
\norm{ \vrh - \vr }_{L^\infty L^2} + \norm{ \vr_h \vu_h - \vr \vu }_{L^\infty L^2}  + \norm{\vuh-\vu}_{L^2L^2} \aleq c(\TS^\frac12 + h^\frac12).
\end{equation*}
\end{thm}

\section{Conclusion}

In this paper we have presented improved error estimates for two well-known numerical methods applied to compressible Navier--Stokes equations. Specifically, we consider the upwind finite volume method and the Marker-and-Cell (MAC) method with implicit time discretization and piecewise constant approximation in space. However, the approach presented in the paper can be applied also to other well-known numerical methods for compressible Navier--Stokes equations.

The novelty of our approach lies in the use of continuous form of the relative energy inequality combined with a refined consistency analysis. Thus, following the framework of the Lax equivalence theorem it suffices to show the (energy) stability, cf.~Lemma~\ref{thm_stability}, and  the consistency of a numerical scheme, cf.~Lemma~\ref{thm_CS}, in order to obtain the convergence rates for the scheme. Indeed, the consistency errors directly yield global errors in the relative energy. To obtain the corresponding error estimates we only assume that the initial data are sufficiently regular and a strong solution exists globaly in time.
The error estimates  presented in Theorem~\ref{thm_EE} improves the results already presented in the literature \cite{Gallouet_mixed, GallouetMAC, MS_MAC}, see Remarks~2,3 for a detailed discussion. In particular, our error estimates hold for the full range of the adiabatic coefficient $\gamma > 1.$

Moreover, we have considered a natural hypothesis on uniformly bounded numerical solutions and proved that the error estimates can be further improved, cf.~Theorem~\ref{thm_EEB}. Indeed, we prove that both numerical methods converge with the first order in time and mesh parameter in terms of the relative energy and with the half order in the $L^\infty(0,T; L^2(\Td))$-norm for the density and momentum, as well as in the  $L^2((0,T) \times \Td)$-norm for the velocity.

\appendix
\section*{Appendix}
\section{Proof of the preliminary lemmas}\label{appa}
In this section we present the proofs of Lemmas \ref{L22} -- \ref{NC}.
\begin{proof}[Proof of Lemma \ref{L22}]
First, we calculate
\begin{align*}
& \intTd{   r_h \Div \vU}
 =
\sumK r_K \intK{     \Div \vU}
=
\sumK r_K \sumfaceK    \int_\sigma \vU \cdot \vn \dS
\\&=
\sumK r_K \sumfaceK    |\sigma| \Piv \vU \cdot \vn
=
\sumK r_K   |K|  \Divw \Piv \vU
=\intTd{r_h \Divw \Piv \vU}.
\end{align*}
Analogously, we find
\begin{align*}
& \intTd{   \vvh \cdot \Grad \psi }
 =
\sumK \vv_K \cdot \intK{     \Grad \psi }
=
\sumK \vv_K \cdot \Big( \sumfaceK     \int_\sigma \psi  \vn \dS \Big)
\\&=
\sumK \vv_K \cdot  \Big( \sumi \sumfaceiK   |\sigma| \Pivi \psi   \vn \Big)
=
\sumi \sumK \vih|_K  \; \left( |K|  \pdmeshi \Pivi \psi \right)
\\&
=\sumi \intTd{\vih \; \pdmeshi \Pivi \psi}
=\intTd{   \vvh \cdot \Gradpiv \psi },
\end{align*}
which completes the proof.
\end{proof}
\medskip
\begin{proof}[Proof of Lemma \ref{L23}]
First, we calculate
\begin{align*}
& \intTd{   \vuh \cdot \Grad \psi}
=  \sumi   \sum_{\sigma  \in \facei} \int_{D_\sigma} \uih \pd_i  \psi \dx
=  \sumi   \sum_{\sigma  \in \facei}  \uih  \left(  \int_{\epsilon^+}  \psi \dS - \int_{\epsilon^-} \psi \dS \right)
\\&
=  \sumi   \sum_{\sigma  \in \facei}  \uih  \left(  \int_{D_{\epsilon^+}}  \Pid \psi  \dx  - \int_{D_{\epsilon^-}} \Pid \psi  \dx \right)/h,
\end{align*}
where $\epsilon^-$ and $\epsilon^+$ are the left and right edges of $D_\sigma$ in the $\ith$-direction of the canonical system for $\sigma \in \facei$. Note that $D_{\epsilon^\pm} \subset \mesh$ are elements of the primary grid $\mesh$. 
Then we can rewrite the above relation as
\begin{align*}
& \intTd{   \vuh \cdot \Grad \psi}
=   \sumi   \sum_{\sigma  \in \facei}    |D_\sigma|  \uih  \pdedgei  \Pid \psi \dx
=  -  \sumi   \sumK  |K|  \pdmeshi \uih    \Pid \psi \dx
= -  \intTd{ \Pid \psi  \; \Divw \vuh},
\end{align*}
where we have used \eqref{IBP2}. This proves \eqref{IBP3}.

The proof of \eqref{IBP4} follows from  \eqref{Divcd2} and \eqref{IBP2}, specifically,
\begin{align*}
 \intTd{   \vvh \cdot \Grad \psi}
=  \intTd{   \vvh \cdot \Gradpiv \psi}
=  \sumi \intTd{ \vih   \pdmeshi \Pivi \psi }
=  -  \sumi \sumK \intK{ \pdedgei \vih   \; \Pivi \psi }.
\end{align*}
\end{proof}

\begin{proof}[Proof of Lemma \ref{L24}]
First, we recall \eqref{Divcd} and \eqref{IBP2} to derive the first equality
\begin{align*}
& \intTd{  \avuh \cdot \Lap \vU }
 =   \intTd{   \avuh \cdot ( \Div  \Grad \vU )}
 =  \intTd{   \avuh \cdot ( \Divw \Piv \Grad \vU )}
\\&=  -
\intTd{    \GradD \avuh : \Piv \Grad \vU  }
=  - \sumi \sumj  \sum_{\sigma \in \facei} \int_{D_\sigma}  \pdedgei \Ov{\ujh} \Pivi \pd_i U_j  \dx
  \\&=  -
  \sum_{i=1}^d  \sumj  \sum_{\sigma \in \facei} \int_{D_\sigma} \left( \frac12 \sum_{\epsilon \in \Eji(D_\sigma) }(\pdBji \ujh)_{D_\epsilon} \right)   \Pivi \pd_i U_j  \dx
 \\&= -
  \sumi   \sumj  \sum_{\epsilon = D_\sigma|D_{\sigma'} \in \Eji} \int_{D_\epsilon} \pdBji \ujh  \left(  \frac{ (\Pivi \pd_i U_j)_{D_\sigma} +(\Pivi \pd_i U_j)_{D_{\sigma'}} }{2}  \right) \dx
\end{align*}

Next, it is easy to check \eqref{IBP6} by setting $\psi =\Div \vU$ in \eqref{IBP3}, i.e.,
\begin{align*}
\intTd{   \vuh \cdot \Grad \Div \vU }
= -  \intTd{\Divw \vuh \Pid( \Div \vU) }  .
\end{align*}
Further, thanks to \eqref{Divcd} and \eqref{IBP2}, we observe \eqref{IBP7}, i.e.,
\begin{align*}
 \intTd{  \vvh \cdot \Lap \vU }
 =   \intTd{   \vvh \cdot ( \Div  \Grad \vU )}
 =   \intTd{   \vvh \cdot ( \Divw \Piv \Grad \vU )}
=
- \intTd{    \GradD \vvh : \Piv \Grad \vU  }.
\end{align*}
Finally, by setting $(\avs{\vvh},\Div \vU)$ as $(\vuh, \psi)$ into \eqref{IBP3} we get \eqref{IBP9}, i.e.
\begin{equation*}
 \intTd{   \avs{\vvh} \cdot \Grad \Div \vU}
= -  \intTd{ \Pid \Div \vU \; \Divw(\avs{\vvh}) }
=-  \intTd{ \Pid \Div \vU \; \Divq \vvh } ,
\end{equation*}
where we have used the identity \eqref{diveq}.
\end{proof}

\begin{proof}[Proof of Lemma \ref{NC}]
Note that the estimates stated in \eqref{NC2} -- \eqref{NC4} hold due to the standard interpolation error;
whence  we omit the proof.
Now we prove \eqref{NC1}. First, by a direct calculation, we have
\begin{align*}
& \norm{\Piq \vuh -\vuh}^2_{L^2} = \sumK \sumi \sumfaceiK |D_{\sigma,K}| \left(  \frac{u_{i,\sigma_{K,i+}} +u_{i, \sigma_{K,i-} } }{2}  - u_{i,\sigma} \right)^2
\\& =  \frac14 \sumK \sumi   \left(  \frac{u_{i,\sigma_{K,i+}}  - u_{i,\sigma_{K,i-}} }{2} \right)^2 \sumfaceiK |D_{\sigma,K}|
 =  \frac{h^2}4 \sumK  |K|  \sumi \left( \pdmeshi  \uih \right)^2
  \leq \frac{h^2}4 \norm{\GradB \vuh}^2_{L^2},
\end{align*}
where we have used the fact that $\pdBii = \pdmeshi$ in the last inequality, which proves the first estimate of \eqref{NC1}.
Analogously, we compute
\begin{align*}
& \norm{\avs{ \vvh} -\vvh}^2_{L^2} = \sumK \sumi \sumfaceiK |D_{\sigma,K}| \left( \frac{\vih^{\rm in} + \vih^{\rm out} }{2}  - \vih^{\rm in}  \right)^2
\\&= \frac{h^2}4  \sumK \sumi \sumfaceiK |D_{\sigma,K}| (\pdedgei \vih)^2 =  \frac{h^2}4  \sumi \sumfaceinti |D_\sigma| (\pdedgei \vih)^2  \leq \frac{h^2}4 \norm{\GradD \vvh}^2_{L^2},
\end{align*}
which proves the second estimate of \eqref{NC1}. This concludes the proof of Lemma~\ref{NC}.
\end{proof}

\section{Sobolev-Poincar\'e type inequality}
First, we recall \cite[Theorem 17]{FeLMMiSh} for a generalized Sobolev-Poincar\'e inequality.
\begin{lemma}[\cite{FeLMMiSh}]\label{lmSP}
For a structure mesh let $\gamma>1$ and $\vrh \geq 0$ satisfy
\[
0<c_M \leq \intTd{\vrh} \mbox{ and }  \intTd{\vrh^\gamma} \leq c_E,
\]
where $\gamma>1$, $c_M$ and $c_E$ are positive constants.
Then there exists $c=c(c_M, c_E, \gamma)$ independent of $h$ such that
\[
\norm{f_h}_{L^q(\tor)}^2  \leq c\left( \norm{ \Gradh f_h}_{L^2(\tor)}^2  + \intTd{\vrh |f_h|^2 } \right).
\]
\end{lemma}
Now we are ready to show the following lemma.
\begin{lemma}\label{lmSP2}
Under the assumption of Lemma \ref{lmSP} let $(\vrh, \vuh)$ be a solution obtained either by the FV method \eqref{VFV_S} or the MAC method \eqref{MAC_S}. Let $\vU \in  W^{2,\infty}(\tor;\R^d)$,
then there exists $C_1 =C_1( M , E_0, \gamma)>0$ and $C_2=C_2( M , E_0, \gamma, \norm{\Grad \vU}_{L^\infty}, \norm{\vU}_{W^{2,\infty}} )>0$ such that
\begin{align}
\label{unormA}
& \norm{ \vuh - \vU}_{L^2}^2   \leq   C_1 \left( \norm{ \Gradh  \vuh - \Grad \vU)}_{L^2(\tor)}^2  + \intTd{\vrh  |\vuh -\vU|^2 }    \right)
 +  C_2 h^2, \\
 \label{unormB}
& \norm{ \Piq \vuh - \vU}_{L^2}^2   \leq   C_1 \left( \norm{ \Gradh \vuh - \Grad \vUh}_{L^2(\tor)}^2  + \intTd{\vrh  |\vuh -\vU|^2 }    \right)
 +  C_2 h^2 ,
\end{align}
where $ M $ and $E_0$ are the fluid mass and initial energy.
\end{lemma}
\begin{proof}
Firstly, by setting $f_h =\vuh -\vUh$ for some $\vUh$ belonging to the same discrete space as $\vuh$ in Lemma \ref{lmSP} we know that
\[
\norm{\vuh -\vUh}_{L^2(\tor)}^2  \leq C_1 \left( \norm{ \Gradh (\vuh -\vUh)}_{L^2(\tor)}^2  + \intTd{\vrh |\vuh -\vUh|^2 } \right),
\]
where the constant $C_1$ depends on $c_M \equiv  M $, $c_E \equiv E_0$ and $\gamma$. Note that the choices of $c_M$ and $c_E$ are owing to the mass conservation \eqref{MC} and energy stability \eqref{ST}.

Next, for $\vuh \in \Qh$ and $\vuh \in \vWh$ we  set $\vUh = \Piq \vU \in \Qh$ and $\vUh = \Piv \vU \in \vWh$, respectively.  Then by the triangular inequality and projection error we derive
\begin{align*}
& \norm{ \vuh - \vU}_{L^2} ^2
\leq
 \norm{ \vuh - \vUh}_{L^2} ^2 + \norm{ \vUh - \vU}_{L^2}^2
 \\&
 \leq   C_1 \left( \norm{ \Gradh (\vuh -\vUh)}_{L^2(\tor)}^2  + \intTd{\vrh |\vuh -\vUh|^2 } \right)   + \left(h   \norm{\Grad \vU }_{L^2}\right)^2
 \\&
 \leq   C_1 \left( \norm{ \Gradh \vuh - \Grad \vU}_{L^2(\tor)}^2  +  \intTd{\vrh  |\vuh -\vU|^2}  \right)
  \\& \quad  +C_1 \left(  \norm{ \Grad \vU - \Gradh \vUh}_{L^2(\tor)}^2  + \intTd{\vrh |\vUh -\vU|^2 }  \right)   + h^2   \norm{\Grad \vU }_{L^2}^2
 \\&
 \leq   C_1 \left( \norm{ \Gradh \vuh - \Grad \vU }_{L^2(\tor)}^2  + \intTd{\vrh  |\vuh -\vU|^2 }  \right)
 \\& \quad +  C_1 \left( h^2 \norm{ \vU}_{W^{2,\infty}}^2 +  h^2 \norm{\Grad \vU}_{L^\infty}^2  \intTd{\vrh}  \right)   + h^2   \norm{\Grad \vU }_{L^2}^2
  \\&
 =  C_1 \left( \norm{ \Gradh \vuh - \Grad \vU }_{L^2(\tor)}^2  + \intTd{\vrh  |\vuh -\vU|^2 }    \right)
 +  C_2 h^2 ,
\end{align*}
where $C_2$ depends on $C_1, \norm{ \vU}_{W^{2,\infty}},\norm{\Grad \vU }_{L^\infty}, M$, and $ \norm{\Grad \vU }_{L^2}$, which proves \eqref{unormA}.

Finally, we proceed with the proof of \eqref{unormB}.
On the one hand, for  the case of $\vuh \in \Qh$ we have $\Piq \vuh =\vuh$, meaning \eqref{unormB} automatically holds as it is the same as \eqref{unormA}.
On the other hand, for the case of $\vuh \in \vWh$ we employ \eqref{unormA} and the triangular inequality to derive
\begin{align*}
& \norm{\Piq \vuh - \vU}_{L^2}^2  \leq \norm{\Piq \vuh  -\vuh}_{L^2} ^2 +  \norm{ \vuh - \vU}_{L^2} ^2
 \\&
  \leq  h^2  \norm{ \Divh \vuh }_{L^2}^2+  C_1 \left( \norm{ \Gradh  \vuh - \Grad \vU }_{L^2(\tor)}^2  + \intTd{\vrh  |\vuh -\vU|^2 }    \right)
 +  C_2 h^2
 \\&
 \aleq   C_1 \left( \norm{ \Gradh  \vuh - \Grad \vU }_{L^2(\tor)}^2  + \intTd{\vrh  |\vuh -\vU|^2 }    \right)
 +  C_2 h^2,
\end{align*}
where we have used the fact that $\norm{ \Divh \vuh }_{L^2}^2 \aleq E_0$ in view of \eqref{est_u},
which completes the proof.
\end{proof}

Next, we recall \cite[Lemma 14.3]{FeLMMiSh} in order to show the following statement formulated in Lemma~\ref{lmRU}.
\begin{lemma}[\cite{FeLMMiSh}]\label{RL2}
Let $\gamma>1$,  $\underline{r} = \frac12 \min\limits_{(t,x) \in Q_T} r >0 $ and $\overline{r} = 2 \max\limits_{(t,x) \in Q_T} r$.  Then there exists $ C= C(\underline{r}, \overline{ r}) >0$ such that
\begin{equation*}
(\vr -r)^2  1_{\rm ess}(\vr)  + (1+ \vr^\gamma) 1_{\rm res}(\vr) \leq C
\bbE(\vr|r),  
\end{equation*}
where $\bbE(\vr|r) =   P(\vr) - P'(r) (\vr -r) - P(r) $ and
\begin{equation}\label{essres}
(1_{\rm ess}(\vr), 1_{\rm res}(\vr)) =
\begin{cases}
(1,0) & \mbox{ if } \vr \in [\underline{r}, \Ov{r}], \\
(0,1) & \mbox{ if } \vr \in \R^+\backslash[\underline{r}, \Ov{r}].
\end{cases}
\end{equation}
\end{lemma}
Now we are ready to show the following lemma.
\begin{lemma}\label{lmRU}
Let $(\vrh, \vuh)$ be a solution obtained either by the FV method \eqref{VFV_S} or the MAC method \eqref{MAC_S}, and let
$\vU \in  L^\infty( 0,T; W^{2,\infty}( \tor; \R^d ))$. Then there holds
\begin{equation*}\label{RU}
\intTO{| (\vrh -r) (\avuh -\vU) | } 
\leq C_0 \intT{ \frakE(\vrh, \vuh| r, \vU)    }  + C_1 \delta \norm{\Gradh \vuh -\Grad \vU}_{L^2}^2 +C_2\delta h^2,
\end{equation*}
where $C_1$, $C_2$ are the same as in Lemma~\ref{lmSP2}, and $C_0$ depends on $\underline{r}, \overline{ r}, \delta$, $ M , E_0, \gamma$.
\end{lemma}
\begin{proof}
First, thanks to Lemma \ref{RL2} we observe
\begin{align*}
&\intTO{ 1_{\rm res}(\vrh)  \; \vrh }
= \intTO{  1_{ \vrh < \underline{r} }  \;  \vrh } +\intTO{ 1_{ \vrh > \overline{ r}} \;  \vrh }
\\&  \leq
 \underline{r}  \intTO{  1_{ \vrh < \underline{r} }  \; 1 } +\intTO{ 1_{ \vrh > \overline{ r}} \;  \vrh^\gamma }
 \leq C \intT{ \frakE(\vrh, \vuh| r, \vU)    } ,
 \end{align*}
 where $C=C(\underline{r}, \overline{ r})$ is given in Lemma \ref{RL2}.

Next, using the triangular inequality,  Young's inequality, the above estimate, Lemma~\ref{lmSP2} and Lemma~\ref{RL2} we find
\begin{align*}
&\intTO{| (\vrh -r) (\avuh -\vU) | }
\\& \leq
\intTO{1_{\rm ess}(\vrh)  | (\vrh -r) (\avuh -\vU) | }   + \intTO{ 1_{\vrh <\underline{r}}  \Ov{r} |\avuh -\vU | }
\\& \quad +  \intTO{ 1_{\vrh > \overline{ r}}  \vrh |\avuh -\vU | }
\\& \leq
\intTO{1_{\rm ess}(\vrh)  \frac12 \big( (\vrh -r)^2 + \vrh |\avuh -\vU |^2/\underline{r} \big) }
\\& \quad
 + \intTO{  1_{\vrh <\underline{r}}   \frac12 \big( \frac{1}{\delta} \Ov{r}^2  + \delta |\avuh -\vU|^2   \big) }
\\& \quad
 + \intTO{ 1_{\vrh > \overline{ r}} \frac12 \big(  \vrh  + \vrh |\avuh -\vU|^2   \big) }
\\&\aleq
C_0\intT{ \frakE(\vrh, \vuh| r, \vU)    }  + C_1 \delta \norm{\Gradh \vuh -\Grad \vU}_{L^2}^2 +C_2\delta h^2,
\end{align*}
where $C_0$ depends on $\underline{r}$, $C(\underline{r}, \overline{ r})$, $\delta$, and $C_1$.
We have completed the proof.
\end{proof}
\section{Relative energy norm}
In this section we show how to control the errors in the conservative variables by the relative energy.
\begin{lemma}\label{lem_EN}
Let $\gamma>1$ and  $(r, \vU)$ satisfy
 \[  \underline{r} = \frac12 \min\limits_{(t,x) \in Q_T} r , \quad
 \Ov{r} = 2 \max\limits_{(t,x) \in Q_T} r, \quad
  \Ov{U} = \max\limits_{(t,x) \in Q_T} |\vU|
 \]
 for some positive constants $\Ov{u},  \underline{r}, \Ov{r}.$

\begin{itemize}
\item
If  $\vr >0$ and $\intTd{\vr ^\gamma} \leq E_0$ hold,  then
\begin{subequations}\label{ER1}
\begin{equation}\label{ER1D}
\norm{\vr - r}_{L^\gamma} + \norm{\vm - \vM}_{L^{\frac{2\gamma}{\gamma+1}}}
\aleq \left(\frakE(\vr,\vu|r,\vU) \right)^{\frac12 } +  \left(\frakE(\vr,\vu|r,\vU) \right)^{\frac1\gamma}
\mbox{ for } \gamma \leq 2 ;
\end{equation}
\begin{equation}\label{ER1M}
\norm{\vr - r}_{L^2} + \norm{\vm - \vM}_{L^{\frac{2\gamma}{\gamma+1}}}
\aleq \left(\frakE(\vr,\vu|r,\vU) \right)^{\frac12 }
 \mbox{ for } \gamma \geq 2 ,
\end{equation}
\end{subequations}
where $m= \vr \vu$ and $\vM= r \vU$.

\item
In addition, let $ \vr <\Ov{\vr}$. Then
\begin{equation}\label{ER2}
\norm{\vr - r}_{L^2} + \norm{\vm - \vM}_{L^2} \aleq \left(\frakE(\vr,\vu|r,\vU) \right)^{\frac12 } .
\end{equation}

\end{itemize}
\end{lemma}
\begin{proof}
First, by the triangular inequality and Lemma~\ref{RL2} we obtain for $\gamma \leq 2$ that
\begin{equation*}
\begin{aligned}
&
\norm{\vr - r}_{L^\gamma}
\leq \norm{(\vr  - r) 1_{\rm ess}(\vr) }_{L^\gamma} + \norm{( \vr  - r)1_{\rm res}(\vr) }_{L^\gamma}
\aleq \norm{(\vr  - r) 1_{\rm ess}(\vr) }_{L^2} + \norm{( \vr  - r)1_{\rm res}(\vr) }_{L^\gamma}
\\&
\aleq  \big( \bbE(\vr|r)\big)^{1/2} +  \left(  \norm{\vr }_{L^\gamma}  + \norm{ r}_{L^\gamma} \right)1_{\rm res}(\vr)
\aleq    \big(\bbE(\vr|r)\big)^{1/2} +  \left(  \intTd{\vr^\gamma \; 1_{\rm res}(\vr) }\right)^{1/\gamma}  + \left( \intTd{1_{\rm res}(\vr) }\right)^{1/\gamma}
\\&
\aleq  \big(\bbE(\vr|r) \big)^{1/2}+  \big(\bbE (\vr|r)\big)^{1/\gamma}
\leq  \big(\frakE(\vr,\vu|r,\vU)\big)^{1/2}  +  \big(\frakE(\vr,\vu|r,\vU)\big)^{1/\gamma} ,
\end{aligned}
\end{equation*}
where $1_{\rm ess}(\vr)$  and $1_{\rm res}(\vr)$  are given in Lemma~\ref{RL2}.
Further, utilizing the above estimate with  the triangular inequality, H\"older's inequality, and the $L^\gamma$ bound on $\vr$, we find
\begin{align*}
&
\norm{\vm - \vM}_{L^{\frac{2\gamma}{\gamma+1}}   }
\leq
\norm{\vr (\vu - \vU)}_{L^{\frac{2\gamma}{\gamma+1}}   }
+\norm{(\vr -r) \vU}_{L^{\frac{2\gamma}{\gamma+1}}   }
\aleq
\norm{ \sqrt{\vr} }_{L^{2\gamma}  }  \norm{ \sqrt{\vr} (\vu - \vU)}_{L^2}
+\norm{\vr -r}_{L^\gamma  }\norm{\vU}_{L^{\frac{2\gamma}{\gamma-1}}   }
\\&
\aleq
\norm{\vr }_{L^\gamma }^{1/2}  \norm{ \vr |\vu - \vU|^2}_{L^1}^{1/2}
+\norm{\vr -r}_{L^\gamma}\norm{\vU}_{L^\infty  }
 \aleq  \big(\frakE(\vr,\vu|r,\vU)\big)^{1/2}  +  \big(\frakE(\vr,\vu|r,\vU)\big)^{1/\gamma}
\end{align*}
which proves \eqref{ER1D}.

Next, again by the triangular inequality and Lemma~\ref{RL2} we observe for $\gamma \geq 2$ that
\begin{equation*}
\begin{aligned}
&
\norm{\vr - r}_{L^2}
\leq \norm{(\vr - r)1_{\rm ess}(\vr) }_{L^2} + \norm{(\vr  - r)1_{\rm res}(\vr) }_{L^2}
\\&
\aleq  \big( \bbE(\vr|r)\big)^{1/2} +  \left( \intTd{ \vr^2 \;1_{\vr > \Ov{r}}   } \right)^{1/2}+  \left( \intTd{ 1_{\rm res}(\vr) } \right)^{1/2}
\\&
\aleq    \big(\bbE(\vr|r)\big)^{1/2} +  \left(  \intTd{\vr^\gamma \; 1_{\vr>\Ov{r} } }\right)^{1/2}
\aleq  \big(\bbE(\vr|r) \big)^{1/2},
\end{aligned}
\end{equation*}
where we have used the fact that $\vr^2 \leq  \vr^\gamma$ for large $\vr$ with $\gamma \geq 2$. Further, it is easy to check that
\begin{align*}
&
\norm{\vm - \vM}_{L^{\frac{2\gamma}{\gamma+1}}   }
\leq
\norm{\vr (\vu - \vU)}_{L^{\frac{2\gamma}{\gamma+1}}   }
+\norm{(\vr -r) \vU}_{L^{\frac{2\gamma}{\gamma+1}}   }
\\&
\aleq
\norm{ \sqrt{\vr} }_{L^{2\gamma}  }  \norm{ \sqrt{\vr} (\vu - \vU)}_{L^2}
+\norm{\vr -r}_{L^2  }\norm{\vU}_{ L^{2\gamma} }
\\&
\aleq
\norm{\vr }_{L^\gamma }^{1/2}  \norm{ \vr |\vu - \vU|^2}_{L^1}^{1/2}
+\norm{\vr -r}_{L^2}\norm{\vU}_{L^\infty  }
 \aleq  \big(\frakE(\vr,\vu|r,\vU)\big)^{1/2}
\end{align*}
which proves \eqref{ER1M}.

When assuming an upper bound on $\vr$, we derive via Lemma~\ref{RL2} that
\begin{equation*}
\norm{\vr - r}_{L^2}
\leq \norm{(\vr - r)1_{\rm ess}(\vr) }_{L^2} + \norm{(\vr  - r)1_{\rm res}(\vr) }_{L^2}
\aleq  \big( \bbE(\vr|r)\big)^{1/2} +   \norm{1_{\rm res}(\vr) }_{L^2}
\aleq    \big(\bbE(\vr|r)\big)^{1/2}
\end{equation*}
which implies
\begin{align*}
&
\norm{\vm - \vM}_{L^2}
\leq
\norm{\vr (\vu - \vU)}_{L^2  }
+\norm{(\vr -r) \vU}_{L^2 }
\norm{ \sqrt{\vr} }_{L^\infty }  \norm{ \sqrt{\vr} (\vu - \vU)}_{L^2}
+\norm{\vr -r}_{L^2  }\norm{\vU}_{ L^\infty}
\\&
 \aleq  \big(\frakE(\vr,\vu|r,\vU)\big)^{1/2} .
\end{align*}
Combining the above two estimates we get \eqref{ER2} and complete the proof.
\end{proof}

\section{Derivation of the relative energy}\label{appd}
In this section we show the relative energy inequality \eqref{RE1}. We start with the reformulation of the relative energy.
\begin{align*}
\frakE(\vrh, \vu_h| r, \vU) = \intTd{\left( \frac12 \vrh \abs{\avuh- \vU}^2 +  \Hc(\vrh) - \Hc'(r) (\vrh -r ) -\Hc(r) \right)} =
\sum_{i=1}^4 T_i ,
\end{align*}
where
\begin{align*}
&T_1 = \intTdB{ \frac12 \vrh \abs{\avuh}^2 +  \Hc(\vrh) },
&& T_2= \intTd{\vrh \left(\frac12  \abs{\vU}^2  - \Hc'(r)   \right)},
\\&
T_3=-\intTd{\vrh \avuh \cdot \vU},
&&T_4= \intTdB{ r \Hc'(r) - \Hc(r) }.
\end{align*}

Next, 
we collect the energy estimate \eqref{ST}, and set the test function $\phi = \left( \frac12  \abs{\vU}^2 -\Hc'(r) \right)$ in the consistency formulation \eqref{CS1}, as well as $\bfphi=- \vU$ in the consistency formulation \eqref{CS2} to get respectively the following
\begin{align*}
\left[ T_1 \right]_{t=0}^\tau =\left[ \intTdB{ \frac12 \vrh \abs{\avuh}^2 +  \Hc(\vrh) } \right]_{t=0}^\tau
 \leq
- \mu \intTO{|\Gradh \vuh|^2} -  \nu \intTO{|\Divh \vuh|^2},
\end{align*}
\begin{align*}
& \left[ T_2 \right]_{t=0}^\tau = \left[ \intTd{ \vrh \underbrace{\left(\frac12  \abs{\vU}^2  - \Hc'(r)   \right) }_{\text{ test function in } \eqref{CS1} } }\right]_{t=0}^\tau
\\&  = \intTOB{\vrh  \pdt  \frac{\abs{\vU}^2}2 + \vrh \avuh \cdot \Grad \frac{ \abs{\vU}^2}2 } +  { e_\vr \left( \tau, \TS, h,{\abs{\vU}^2/2}\right)}
\\& - \intTOB{\vrh   \pdt \Hc'(r) + \vrh \avuh \cdot \Grad \Hc'(r)} -   { e_\vr( \tau, \TS, h,{ \Hc'(r) })},
\end{align*}
\begin{align*}
&\left[ T_3 \right]_{t=0}^\tau = \left[ \intTd{\vrh \avuh \cdot \underbrace{( - \vU)}_{\text{ test function in } \eqref{CS2} }    }  \right]_{t=0}^\tau
\\& = -  \intTOB{\vrh \avuh \cdot \pdt \vU + \vrh \avuh \otimes \avuh : \Grad \vU + p_h \Div \vU}
\\& + \intTOB{ \mu  \Gradh \vuh: \Grad \vU + \nu \Divh \vuh \Div \vU}  +  {e_{\vm} ( \tau, \TS, h,-\vU)}.
\end{align*}
Moreover, the term $T_4$ reads
\begin{align*}
\left[ T_4 \right]_{t=0}^\tau = \left[ \intTdB{ r \Hc'(r) - \Hc(r) } \right]_{t=0}^\tau = \intTO{ \pdt( r \Hc'(r) - \Hc(r)) }.
\end{align*}
Summing up the above terms we get \eqref{RE1}
\begin{equation*}
\begin{aligned}
& \left[ \frakE(\vrh, \vu_h| r, \vU)  \right]_0^{T} +  \intTOB{\mu \abs{ \Gradh \vuh}^2 +  \nu \abs{\Divh \vuh}^2 }
\\& \leq \intTOB{\vrh  \pdt  \frac{\abs{\vU}^2}2 + \vrh \avuh \cdot \Grad \frac{ \abs{\vU}^2}2 } +  { e_\vr \left( \tau, \TS, h,{\abs{\vU}^2/2}\right)}
\\& - \intTOB{\vrh   \pdt \Hc'(r) + \vrh \avuh \cdot \Grad \Hc'(r)} -   { e_\vr( \tau, \TS, h,{ \Hc'(r) })}
\\& - \intTOB{\vrh \avuh \cdot \pdt \vU + \vrh \avuh \otimes \avuh : \Grad \vU + p_h \Div \vU}
\\& + \intTOB{ \mu  \Gradh \vuh: \Grad \vU + \nu \Divh \vuh \Div \vU}  +  {e_{\vm} ( \tau, \TS, h,-\vU)}
\\&
+ \intTO{ \pdt( r \Hc'(r) - \Hc(r)) }.
\end{aligned}
\end{equation*}


\medskip

\begin{thebibliography}{10}




%
%


%


\bibitem{BrFeHo2016}
D.~Breit, E.~Feireisl, and M.~Hofmanov{\'a}.
\newblock Local strong solutions to the stochastic compressible
  {N}avier--{S}tokes system.
\newblock {\em Comm. Partial Differential Equations}, {\bf{43}}(2):313--345, 2018.

\bibitem{DoFei}
V.~Dolej\v{s}\'i and M.~Feistauer.
\newblock {\em Discontinuous Galerkin Method.}
\newblock{Volume 48 of Springer Series in Computational Mathematics}, Springer, 2015.



\bibitem{feist1}
M.~Feistauer.
\newblock {\em Mathematical Methods in Fluid Dynamics.}
\newblock Volume 67 of Pitman Monographs and Surveys in Pure and Applied Mathematics,  Longman Scientific \& Technical, Harlow, 1993.

\bibitem{feist2}
M.~Feistauer, J.~Felcman, I.~Stra\v{s}kraba.
\newblock {\em Mathematical and Computational Methods for Compressible Flow.}
The Clarendon Press, Oxford University Press, 2003.

%
%
\bibitem{EyGaHe}
R.~Eymard, T.~Gallou{\"{e}}t, and R.~Herbin.
\newblock Finite volume methods.
\newblock{\em Handbook of numerical analysis} {\bf 7}: 713--1018, 2000.






\bibitem{Hosek}
E.~Feireisl, R.~Ho\v{s}ek, D.~Maltese, and A.~Novotn{\'y}.
\newblock Error estimates for a numerical method for the compressible Navier--Stokes system on sufficiently smooth domains.
\newblock {\em ESAIM Math. Model. Numer. Anal.} {\bf 51}(1): 279--319, 2017.



\bibitem{edo}
E.~Feireisl.
\newblock {\em Dynamics of Viscous Compressible Fluids.}
\newblock Oxford Lecture Series in Mathematics and its Applications, Oxford University Press, 2004.


\bibitem{FJN}
E.~Feireisl, B.J.~Jin, and A.~Novotn{\'y}.
\newblock Relative entropies, suitable weak solutions, and weak strong uniqueness for the compressible Navier--Stokes system.
\newblock {\em J. Math. Fluid Mech.} {\bf 14}(4): 717--730, 2012.




%
\bibitem{feireisl2017numericsbook}
E.~Feireisl, T.~G.~Karper, and M.~Pokorn{\'y},
\newblock {\em Mathematical Theory of Compressible Viscous Fluids: {A}nalysis
  and {N}umerics.}
\newblock {Birkh{\" a}user--Verlag, Basel} 2017.



\bibitem{FL}
E.~Feireisl, and M.~Luk\'a\v{c}ov\'a-Medvi{\softd}ov\'a.
\newblock Convergence of a mixed finite element--finite volume  scheme for the isentropic Navier--Stokes system via the dissipative measure--valued solutions.
\newblock{\em Found. Comput. Math.} {\bf 18}(3): 703--730, 2018.


%
\bibitem{FLMS_FVNS}
E.~Feireisl, M.~Luk\'a\v{c}ov\'a-Medvi{\softd}ov\'a, H.~Mizerov\'a, and B.~She.
\newblock Convergence of a finite volume scheme for the compressible Navier--Stokes system.
\newblock  {\em ESAIM: M2AN} {\bf 53}(6): 1957--1979, 2019.


\bibitem{FeLMMiSh}
E.~Feireisl, M.~Luk\'a\v{c}ov\'a--Medvi{\softd}ov\'a, H.~Mizerov\'a, and B.~She.
\newblock {\em Numerical Analysis of Compressible Fluid Flows.}
\newblock  Springer--Verlag, 2021.







%

\bibitem{Gallagher}
I. Gallagher.
\newblock  A remark on smooth solutions of the weakly compressible periodic Navier--Stokes equations.
\newblock {\em J. Math. Kyoto Univ.} {\bf 40}(3): 525--540, 2000.





\bibitem{GallouetMAC}
T.~Gallou{\"e}t, D.~Maltese, and A.~Novotn{\'y}.
\newblock Error estimates for the implicit MAC scheme for the compressible Navier--Stokes equations.
\newblock {\em Numer. Math. }  {\bf 141}: 495--567,  2019.

\bibitem{Gallouet_mixed}
T.~Gallou{\"e}t, R.~Herbin, D.~Maltese, and A.~Novotn{\'y}.
\newblock Error estimates for a numerical approximation to the compressible barotropic  Navier--Stokes equations.
\newblock {\em IMA J. Numer. Anal. }  {\bf 36}(2): 543--592,  2016.




\bibitem{HS_MAC}
R.~Ho\v{s}ek and B.~She.
\newblock Stability and consistency of a finite difference scheme for compressible viscous isentropic flow in multi-dimension.
\newblock {\em J. Numer. Math.} {\bf 26}(3): 111--140, 2018.

\bibitem{Jovanovic}
V. Jovanovi\'{c}.
\newblock An error estimate for a numerical scheme for the compressible Navier--Stokes system.
\newblock {\em Kragujevac J. Math. }{\bf 30}: 263--275, 2007.



\bibitem{Karper}
T.~Karper.
\newblock A convergent FEM-DG method for the compressible Navier--Stokes equations.
\newblock {\em Numer. Math.} {\bf 125}(3): 441--510, 2013.



\bibitem{NOKW}
Y.~Kwon and A.~Novotn{\'y}.
\newblock {Consistency, convergence and error estimates for a mixed finite element--finite volume scheme to compressible Navier--Stokes equations with general inflow/outflow boundary data.}
\newblock  {\em IMA J. Numer. Anal.} \textbf{42}(1): 107--164, 2022.

%






\bibitem{Lions}
P.L.~Lions.
\newblock  Mathematical topics in fluid mechanics. Vol. 2: Compressible models. Oxford University Press, 1998.



\bibitem{Liu1}
B. Liu.
\newblock The analysis of a finite element method with streamline diffusion for the compressible Navier--Stokes equations.
\newblock {\em SIAM J. Numer. Anal.} {\bf 38}:1--16, 2000.



\bibitem{Liu2}
B. Liu.
\newblock On a finite element method for three-dimensional unsteady compressible viscous flows.
\newblock {\em Numer. Methods Partial Differ. Eq.} {\bf 20}: 432--449, 2004.




\bibitem{MS_MAC}
H.~Mizerov\'{a} and B.~She.
\newblock Convergence and error estimates for a finite difference scheme for the multi-dimensional compressible Navier--Stokes system.
\newblock {\em J. Sci. Comput.} {\bf 84}(1): No.25, 2020.








%

\bibitem{SuWaZh}
Y.~Sun, C.~Wang, and Z.~Zhang.
\newblock A {B}eale--{K}ato--{M}ajda blow--up criterion for the 3-D compressible Navier--Stokes equations.
\newblock {\em J. Math. Pures. Appl. } {\bf 95}(1): 36--47, 2011.





%
\bibitem{VaZa}
A. Valli and M. Zajaczkowski.
\newblock Navier--Stokes equations for compressible fluids: Global existence and qualitative properties of the solutions in the general case.
\newblock  {\em Commun. Math. Phys.} {\bf 103}: 259--296,1986.

\bibitem{PV}
P.I.~Plotnikov and  W.~Weigant.
\newblock Isothermal Navier--Stokes equations and Radon transform.
\newblock {\em SIAM J. Math. Anal.} \textbf{47}(1): 626–653, 2015.

\bibitem{toro}
E.F.~Toro.
\newblock{\em Riemann Solvers and Numerical Methods for Fluid Dynamics. A practical introduction.}
Springer,  2009.


\end{thebibliography}
\bibliographystyle{plain}

\end{document}